%% file: main.tex
\pdfoutput=1

\documentclass{article}

\usepackage{microtype}
\usepackage{graphicx}
\usepackage{subfigure}
\usepackage{booktabs} 

\usepackage{enumitem}
\usepackage{amssymb}
\usepackage{amsmath,amsthm}

\usepackage{mathrsfs}		
\usepackage{dsfont}

\usepackage{url,hyperref}

\usepackage[accepted]{icml2025}

\renewcommand{\cite}{\citep} 

\hypersetup{colorlinks=true, urlcolor=blue, linktoc = all}

\usepackage{hhline}
\usepackage{multirow}

\usepackage[utf8]{inputenc} 
\usepackage[T1]{fontenc}

\usepackage{multicol}
\usepackage{silence}
\usepackage{amsfonts,thmtools,thm-restate}
\usepackage{mathtools}
\usepackage{xparse}
\usepackage{enumitem} 
\usepackage{etoolbox}
\usepackage{mathtools}
\usepackage{complexity}
\usepackage{svg}
\usepackage{xcolor, soul}
\usepackage{tablefootnote}
\usepackage[bb=boondox]{mathalpha} 
\definecolor{lightgrey}{rgb}{0.9,0.9,0.9}
\sethlcolor{lightgrey}

\definecolor{mygray}{rgb}{0.6,0.6,0.6}

\usepackage[capitalise,nameinlink,noabbrev]{cleveref}  
\crefname{equation}{}{}

\theoremstyle{plain}
\newtheorem{theorem}{Theorem}[section]
\newtheorem{proposition}[theorem]{Proposition}
\newtheorem{lemma}[theorem]{Lemma}
\newtheorem{corollary}[theorem]{Corollary}
\theoremstyle{definition}

\newtheorem{remark}[theorem]{Remark}

\usepackage{tikz}

\usepackage{pgfplots}
\pgfplotsset{minor tick style={draw=none}}
\usepackage{float}

\usepackage{times}
\usepackage{xspace}

\definecolor{labelkey}{rgb}{0,0.08,0.45}
\definecolor{refkey}{rgb}{0,0.6,0.0}
\definecolor{Brown}{rgb}{0.45,0.0,0.05}
\definecolor{dgreen}{rgb}{0.00,0.49,0.00}
\definecolor{dblue}{rgb}{0,0.08,0.75}
\definecolor{ffwwqq}{rgb}{1.,0.4,0.}
\definecolor{qqzzqq}{rgb}{0.,0.6,0.}
\definecolor{qqqqff}{rgb}{0.,0.,1.}
\definecolor{dred}{HTML}{D90404}
\definecolor{orng}{HTML}{D35400}
\definecolor{cb-black}      {RGB}{  0,   0,   0}
\definecolor{cb-blue-green} {RGB}{  0,  073,  073}
\definecolor{cb-green-sea}  {RGB}{  0, 146, 146}
\definecolor{cb-rose}       {RGB}{255, 109, 182}
\definecolor{cb-salmon-pink}{RGB}{255, 182, 119}
\definecolor{cb-purple}     {RGB}{ 73,   0, 146}
\definecolor{cb-blue}       {RGB}{ 0, 109, 219}
\definecolor{cb-lilac}      {RGB}{182, 109, 255}
\definecolor{cb-blue-sky}   {RGB}{109, 182, 255}
\definecolor{cb-blue-light} {RGB}{182, 219, 255}
\definecolor{cb-burgundy}   {RGB}{146,   0,   0}
\definecolor{cb-brown}      {RGB}{146,  73,   0}
\definecolor{cb-clay}       {RGB}{219, 209,   0}
\definecolor{cb-green-lime} {RGB}{ 36, 255,  36}
\definecolor{cb-yellow}     {RGB}{255, 255, 109}
\definecolor{bred}{HTML}{FF0000}
\definecolor{bpurp}{HTML}{BF00BF}
\definecolor{bblu}{HTML}{0000FF}
\definecolor{bcyan}{HTML}{00BFBF}
\definecolor{byellow}{HTML}{BFBF00}
\definecolor{bgreen}{HTML}{008000}

\hypersetup{pdfborder={0 0 0}}

\newcommand{\norm}[1]{\| #1 \|}

\newcommand*\circledaux[1]{\tikz[baseline=(char.base)]{
    \node[shape=circle,draw,inner sep=0.8pt] (char) {#1};}}

\NewDocumentCommand{\circled}{m o }{%
    \IfNoValueTF{#2}{\circledaux{#1}}{\stackrel{\circledaux{#1}}{#2}}%
}

\newcommand{\defi}{\stackrel{\mathrm{\scriptscriptstyle def}}{=}}
\newcommand{\defiin}{\stackrel{\mathrm{\scriptscriptstyle def}}{\in}}

\renewcommand*\R{\mathbb{R}}

\let\epsilon\varepsilon
\let\doteq\defi

\newcommand{\ones}{\mathds{1}}

\usepackage{pifont}

\DeclareMathOperator*{\argmax}{arg\,max}                
\DeclareMathOperator*{\argmin}{arg\,min}

\newcommand{\innp}[1]{\langle #1 \rangle}
\newcommand{\bigo}[1]{O( #1 )}
\newcommand{\bigol}[1]{O\left( #1 \right)}

\input{definitions.tex}
\input{algorithm_config.tex}

\begin{document}

\twocolumn[

\icmltitle{Beyond Short Steps in Frank-Wolfe Algorithms}

\begin{icmlauthorlist}
\icmlauthor{David Martínez-Rubio}{zib,c3u} 
\icmlauthor{Sebastian Pokutta}{tub,zib} 
\end{icmlauthorlist}

\icmlaffiliation{zib}{Zuse Institute Berlin, Germany}
\icmlaffiliation{c3u}{Signal Theory and Communications Department,
  Carlos III University, Madrid, Spain}
\icmlaffiliation{tub}{Institute of Mathematics,
  Technische Universität Berlin, Germany}

\icmlcorrespondingauthor{David Martínez-Rubio}{dmrubio@ing.ucm.es}
\icmlcorrespondingauthor{Gábor Braun}{braun@zib.de}
\icmlcorrespondingauthor{Sebastian Pokutta}{pokutta@zib.de}

\icmlkeywords{Frank—Wolfe algorithm, approximate duality gap}

\vskip 0.3in
]

\printAffiliationsAndNotice{}  

\begin{abstract}
    We introduce novel techniques to enhance Frank-Wolfe algorithms by leveraging function smoothness beyond traditional short steps. Our study focuses on Frank-Wolfe algorithms with step sizes that incorporate primal-dual guarantees, offering practical stopping criteria. We present a new Frank-Wolfe algorithm utilizing an optimistic framework and provide a primal-dual convergence proof. Additionally, we propose a generalized short-step strategy aimed at optimizing a computable primal-dual gap. Interestingly, this new generalized short-step strategy is also applicable to gradient descent algorithms beyond Frank-Wolfe methods. As a byproduct, our work revisits and refines primal-dual techniques for analyzing Frank-Wolfe algorithms, achieving tighter primal-dual convergence rates. Empirical results demonstrate that our optimistic algorithm outperforms existing methods, highlighting its practical advantages.
\end{abstract}

\section{Introduction}
We are interested in solving the following optimization problem:
\[
    \min_{x\in\X} f(x),
\]
where \(\X\) is a compact convex set and $f$ is a convex and $L$-smooth function. The Frank-Wolfe (\newtarget{def:acronym_frank_wolfe}{\FW{}}) algorithm \citep{frank1956algorithm}, also known as the conditional gradient algorithm \citep{levitin1966constrained}, is a key algorithm for this problem class, particularly for problems where projection onto a constraint set is computationally expensive, as it does not require projections. It leverages a linear minimization oracle (\newtarget{def:acronym_linear_minimization_oracle}{\LMO{}}) for \(\X\), which upon presentation with a linear function $c$ returns $v \leftarrow \arg\min_{v\in\X} \innp{c, v}$, and a gradient oracle for \(f\), which given a point $x \in \dom(f)$ returns $\nabla f(x)$. Given its low cost per iteration it is often highly effective in various machine learning applications. These include optimal transport \citep{luise2019sinkhorn}, neural network pruning \citep{ye2020good}, adversarial attacks \citep{chen2020fwframework}, non-negative matrix factorization \citep{nguyen2022memory}, particle filtering \citep{lacoste2015sequential}, and distributed learning \citep{bellet2015distributed}, among others.

A more general version of the Frank-Wolfe algorithm can handle the optimization of \(f(x) + \psi(x)\) for a convex function \(\psi\). In this case however, we require access to a gradient oracle for \(f\) and an oracle that can solve \(\argmin_v\{\innp{w, v} + \psi(v)\}\) for any \(w \in \mathbb{R}^d\), as discussed in \citep{nesterov2018complexity}. We recover the classical \FW{} algorithm by choosing \(\psi(x)\) as the indicator function of the set \(\X\).

In this work, we will focus on primal-dual analyses for \FW{} algorithms. In contrast to classical analyses where a step-size strategy and corresponding convergence rate need to be heuristically estimated and then proven by induction, in the primal-dual setup they emerge naturally as a natural consequence of the analysis. This approach does not only provide tighter dual gaps and hence stopping criteria but also enhances the convergence properties of the algorithm. In particular, rather than considering primal progress without clear indication how good the solution is, primal-dual progress refers to the reduction in the primal-dual gap, which is a measure of how close the current solution is to optimality; this will be our measure here. A well-known strategy for \FW{}  algorithms is the so-called short step, which essentially chooses the step size to maximize primal guaranteed progress from the descent lemma, see \citep{frank1956algorithm}); and also \citep{CGFWSurvey2022} for an in-depth discussion. 

We introduce the concept of primal-dual short steps, that maximize primal-dual guaranteed progress based on a model obtained from the primal-dual analysis. We also show that this new step-size strategy is also applicable to gradient descent algorithms. Moreover, recent advancements have introduced adaptive step-size strategies, such as, e.g., \citep{pedregosa2020linearly} (see also \citep{P2023} for a numerically improved version), which further refine the short-step. These strategies aim to decrease a model of the function adaptively, enhancing the algorithm's performance and stability and in particular allow for leveraging local curvature information. However, they come at the extra cost of a line search like procedure and are subject to the same lower bounds as the classical short steps \citep{guelat1986some}. 

\subsection*{Contribution} Our contribution can be summarized as follows:

\paragraph{Optimistic Frank-Wolfe Algorithm} We introduce a novel Frank-Wolfe algorithm that leverages the concept of optimism, as discussed in \citep{rakhlin2013online,steinhardt2014}. This algorithm is designed to adapt effectively to varying conditions and provides a robust analysis of the convergence rate associated with the primal-dual gap. 

\paragraph{Primal-Dual Short Steps} We propose a new class of step-size rules for existing Frank-Wolfe algorithms, termed primal-dual short steps. These steps are based on the sequential minimization of a primal-dual gap defined in the algorithm's analysis at each iteration. This approach is flexible and extends to gradient descent algorithms, allowing for line search over the primal-dual gap in both algorithm classes.

\paragraph{Numerical Experiments} We conduct preliminary numerical experiments to demonstrate the practical advantages of the proposed optimistic Frank-Wolfe algorithm. In our tests, the results consistently show the effectiveness of this new variant in minimizing the primal-dual gap more efficiently than traditional methods. In particular, in most experiments the empirical order of convergence is better. 

Additionally, we have obtained primal-dual analyses for algorithms and results where so far only classical primal convergence analysis where available, often improving constants. Due to space constraints, these findings are detailed in \cref{sec:pd_extensions_fw}.

\subsection*{Related Work}
The Frank-Wolfe algorithm, originally introduced by \citet{frank1956algorithm} where also short steps were introduced, is a foundational method in projection-free optimization, particularly for problems involving convex constraints. \citet{levitin1966constrained} further extended the method, which they referred to as \textit{conditional gradient algorithm}. A significant advancement in the understanding of the Frank-Wolfe algorithm was made by \citet{jaggi2013revisiting}, who introduced the concept of the \FW{} gap and provided convergence guarantees on this measure, marking the first instance of \FW{} guarantees for a primal-dual gap.

The development of primal-dual methods for conditional gradient algorithms has been explored in works such as \citet{nesterov2018complexity} and \citet{diakonikolas2019approximate}, which have contributed to the broader understanding of these algorithms. Additionally, \citet{abernethy2017fwequilibrium} described a Frank-Wolfe variant that incorporates cumulative gradients, albeit with equal weights, resulting in an additional logarithmic factor in complexity.

Several methods apply the \LMO{} at a convex combination of previously obtained directions. In particular, variants of the Frank-Wolfe algorithm have been developed by \citet{mokhtari2020stochastic}, \citet{lu2021generalized}, and \citet{negiar2020stochastic} using stochastic gradients. Notably, the analysis of \citet{lu2021generalized} reduces to the optimal rate in the deterministic case, up to constant factors. \citet{wang2024noregret} also presented a cumulative-gradient \FW{} algorithm, with additional logarithmic factors. Further contributions include the work of \citet{li2021heavy}, who provided a primal-dual analysis of the deterministic algorithm, heavy-ball Frank-Wolfe (\newtarget{def:acronym_heavy_ball_frank_wolfe}{\HB{}}), using a convex combination of gradients. 

In the context of online learning, the concept of optimism, which involves using hints such as past losses, to predict future losses, originated in \citet{azoury2011relative}, \citet{chiang2012online}, and \citet{rakhlin2013online}. This idea had already been developed in saddle point optimization, as seen in the work of \citet{popov1980modification}. 

A duality was found and further explored between mirror descent and Frank-Wolfe algorithms by \citet{bach2015duality} and \citet{pena2019generalized}. \citet{wang2024noregret} and \citet{gutman2023perturbed} have also shown that Frank-Wolfe and mirror descent, among other methods, can be unified under a common framework. In this work, we make use and generalize an anytime online-to-batch conversion \citep{nesterov2015quasi,cutkosky2019anytime} for connecting Frank-Wolfe algorithms with online learning ones.

Works that have studied modifications of the gradients to obtain faster convergence of \FW{} algorithms and linking them back to gradient descent methods and gradient mappings include \citet{CP2020boost,MGP2020walking} and in \citet{CDP2019} local acceleration in the Nesterov sense has been demonstrated for \FW{} algorithms for strongly convex functions while whether local acceleration is possible for non-strongly convex but smooth functions remains an open question. 

\section{Preliminaries and Groundwork}\label{sec:preliminaries}

In this section, we introduce the necessary notation and an overview of some \FW{} algorithms. In the following let $\X$  be a compact convex set with diameter $D$. Further, let $f: \X \to \R$ be a differentiable function defined on an open set containing $\X$. We assume that $f$ is convex and $L$-smooth on $\X$ with respect to a norm $\norm{\cdot}$. This means that for all $x, y \in \X$, the following inequality holds:
\[
   0 \leq  f(y) - f(x) - \innp{\nabla f(x), y-x} \leq \frac{L}{2}\norm{x-y}^2.
\]
These two conditions imply $f$ has an $L$-Lipschitz gradient $\norm{\nabla f(x) - \nabla f(y)}_\ast \leq L\norm{x-y}$, where $\norm{\cdot}_\ast$ denotes the dual norm.

If not stated otherwise let $x^\ast \in \argmin_{x\in\X} f(x)$ be a minimizer of $f$ over $\X$. Further, let $\newtarget{def:indicator_function}{\indicator}(x)$ be the indicator function of the set $\X$ that is $0$ if $x \in \X$ and $+\infty$ otherwise, and let $\ones_{B}$ be the event indicator function that is $1$ if the event $B$ holds and $0$ otherwise. We denote $[T] \defi \{1, 2, \dots, T\}$. We denote by $\partial \psi(x)$ the subdifferential of $\psi$ at $x$ so that if $g\in\partial\psi(x)$, we have $\psi(y) \geq \psi(x) + \innp{g, y-x}$. 

We provide now a sketch of the proof structure for some primal-dual analysis of \FW{} as well as \HB{} and other methods, and in \cref{sec:review_of_primal_dual_fw} we provide full details of this overview. With respect to \citep{li2021heavy}, we provide a slightly more general analysis of \HB{} by allowing for a composite term and we provide an analysis of \FW{} when decreasing regularization is used. We start by defining a lower bound $L_t$ on the optimal value that we have access to at time $t$, and for iterates $x_t$, $t\geq 1$, we define a primal-dual gap as 
\begin{equation}\label{eq:def_primal_dual_gap}
    G_t \defi f - L_t
\end{equation}
where $f(x_{t+1})$ is one step ahead of the lower bound, which is a subtle shift that differs slightly from known analysis and allows us to get very clean primal-dual analyses of and convergence results for the new and old algorithms we present in this work. In the sequel, if not stated otherwise, let $a_t > 0$ to be determined later and define $A_t = A_{t-1} + a_t = \sum_{\ell = 0}^t a_\ell$.
For instance, for \FW{} and FW-HB we use the following lower bounds, or rather, the choice of lower bound defines the iterate $v_t$ of the algorithm:
\begin{align}\label{eq:lower_bounds}
    \begin{aligned}
        & \ A_t f(x^\ast) \\ 
        &\circled{1}[\geq]  \sum_{\ell=0}^t a_\ell f(x_\ell) + \sum_{\ell=0}^t a_\ell \innp{\nabla f(x_\ell), x^\ast - x_\ell}  \\
        &{\circled{2}[\geq]} \sum_{\ell=0}^t a_\ell f(x_\ell) + \min_{v\in\X}\Big\{ \sum_{\ell=0}^ta_\ell \innp{\nabla f(x_\ell), v - x_\ell} \Big\}{\defi}A_t L_t^{\operatorname{HB}}  \\
        &{\circled{3}[\geq]} \sum_{\ell=0}^t a_\ell f(x_\ell) + \sum_{\ell=0}^t a_\ell \min_{v\in\X} \innp{\nabla f(x_\ell), v - x_\ell} \defi A_t L_t^{\operatorname{FW}}\\
        & = \sum_{\ell=0}^t a_\ell f(x_\ell) - \sum_{\ell=0}^t a_\ell g(x_\ell).
    \end{aligned}
\end{align}
where the term above $g(x_\ell) \defi \max_{v\in\X} \innp{\nabla f(x_\ell), x_\ell-v}$ $\left( \geq f(x_\ell) - f(x^\ast) \right)$ is the so-called FW gap, which is also a primal dual gap, the best of which for $\ell \in [t]$ was shown to converge by \citet{jaggi2013revisiting} at a $\bigo{\frac{LD^2}{t}}$ rate for \FW{}. Above $\circled{1}$ holds by convexity, $\circled{2}$ takes a $\min$. This expression is $A_t$ times the lower bound $L_t^{\operatorname{HB}}$ in \HB{} is and its $\argmin$ corresponds to the iterate $v_t$ of \HB{}, cf. \cref{line:heavy_ball_fw}, whereas after $\circled{3}$ we have $A_t$ times the lower bound $L_t^{\operatorname{FW}}$ used in the \FW{} algorithm and the $\argmin$ of the $\ell$-th summand is the iterate $v_t$ of \FW{}, cf. \cref{line:vanilla_fw}.

\begin{algorithm}[H]
    \caption{\label{alg:fw} Frank-Wolfe and Heavy-Ball \FW{} algorithms }
    \begin{algorithmic}[1]
        {\footnotesize
        \REQUIRE Function $f$, feasible set $\X$, initial point $x_0$. Weights $a_t, \gamma_t$.
            \FOR{$t=0$ \textbf{to} $T$}
           \STATE {\color{mygray}$\diamond$ Choose either \ref{line:heavy_ball_fw} (\newlinkcolor{def:acronym_heavy_ball_frank_wolfe}{HB-FW}{mygray}) or \ref{line:vanilla_fw} (\newlinkcolor{def:acronym_frank_wolfe}{FW}{mygray}) for the entire run:}
            \STATE $v_t \leftarrow \argmin\limits_{v \in \mathcal{X}} \sum_{i=1}^{t-1} a_i\langle \nabla f(x_i), v \rangle$\label{line:heavy_ball_fw}
            \STATE $v_t \leftarrow \argmin\limits_{v \in \mathcal{X}} \langle \nabla f(x_t), v \rangle $\label{line:vanilla_fw}
            \STATE $x_{t+1} \leftarrow (1 - \gamma_t) x_t + \gamma_t v_t$
            \ENDFOR}
    \end{algorithmic}
\end{algorithm}

If we now show a bound 
\begin{equation}\label{eq:adgt_bound}
    A_t G_t - A_{t-1} G_{t-1} \leq E_t   \text{ for all } t \geq 0,
\end{equation}
then we have $f(x_{t+1}) -f(x^\ast) \leq G_t \leq \frac{1}{A_t}\sum_{i=0}^t E_i$. Our aim is thus to have small $E_t$ and large $A_t$. Note $A_{-1} = 0$ by definition. If $f$ is $L$-smooth with respect to $\norm{\cdot}$ and $D \defi \max_{x, y \in \X}\norm{x-y}$, then choosing $a_t = 2t +2$, $\gamma_t = a_t / A_t$ one can show that both \FW{} and \HB{} in \cref{alg:fw} satisfy \cref{eq:adgt_bound} with $E_t = \frac{LD^2 a_t^2}{2A_t}$ and consequently $G_t \leq \frac{2LD^2}{t+2}$, cf. \cref{sec:review_of_primal_dual_fw}. The term $f(x_{t+1}) - f(x^\ast)$ cannot be computed in general since we usually do not have access to the value $f(x^\ast)$. However, the primal-dual bound $G_t$ is computable and thus it can be used as a stopping criterion. 

\begin{remark}[Alternative step-size strategies]\label{eq:non_open_loop_step_sizes}
    The analysis uses the guaranteed \textit{primal progress}, that is, the descent that is guaranteed on $f(x_{t+1}) - f(x_t)$, in particular with the choice $\gamma_t = \frac{a_t}{A_t} = \frac{2}{t+2}$ by using the upper bound that the smoothness inequality provides. Any step size strategy whose primal progress dominates this guaranteed progress yields a better convergence rate. For instance, the so-called short step-size strategy, maximizes this guaranteed progress along the segment joining $x_t$, and $v_t$. Indeed the best lower bound on $f(x_{t}) - f(x_{t+1})$ given by smoothness is 
    \[
        \max_{\gamma \in [0, 1]}\left\{\gamma\innp{f(x_t), x_t - v_t} + \gamma^2\frac{L\norm{x_t-v_t}^2}{2}\right\}.
    \] 
    And the $\argmax$ is $\gamma = \min\{1, \frac{\innp{\nabla f(x_t), x_t - v_t}}{L\norm{x_t -v_t}^2}\}$, which is the short step-size.
    Another alternative is performing a line search over the segment joining $x_t$ and $v_t$. Even if the line search is performed with an error at iteration $t$, this error only contributes additively to the corresponding $E_t$, and so its global impact can be easily quantified.
\end{remark}

Although short steps induce monotonic primal progress by definition, they do not necessarily induce monotonic primal-dual progress, which is the important measure to look at when we require a stopping criterion. To remedy this, in \cref{sec:primal-dual-short-step} we introduce a new primal-dual short step strategy that induces monotonic primal-dual progress.

On the other hand all of previous FW approaches exploit smoothness of $f$ in the same way: a quadratic upper bound on the function is computed and the yielded guaranteed progress is used to compensate from some other errors in the analysis. In \cref{sec:optimistic_fw}, we introduce a framework that goes beyond this by making use of optimism. The algorithm has the potential of better adapt to the environment, as we show in the experiments section, where Optimisitic Frank-Wolfe performs better than other approaches.

\section{An Optimistic Frank-Wolfe algorithm}\label{sec:optimistic_fw}

In this section, we propose a new Frank-Wolfe method that at each iteration, uses a prediction of the next gradient in order to minimize a suitable regularized lower model of the objective. The more accurate this prediction is, the better the resulting convergence rate. For functions with Lipschitz continuous gradients, the previous gradient serves as a sufficiently accurate predictor for the next one. We sketch the main ideas of our new algorithm and provide all details in \cref{sec:optimistic_proofs}.

The algorithm and analysis is based on optimistic versions of the Online Mirror Descent (OMD) algorithm and of the Follow the Regularized Leader (FTRL) algorithm. We provide a slight generalization over these online learning algorithms in order to allow for non-differentiable regularizers, in order to cover typical cases where Frank-Wolfe is applied. To that effect, we make use of the following Bregman divergence definition, that specifies a subgradient of $\phi$ of the regularizer for its definition, that is,
\[
    D_{\psi}(x, y, \phi) \defi \psi(x) - \psi(y) - \innp{\phi, x-y}.
\] 
where $\phi \in \partial \psi(y)$. Note that in the pseudocode of \cref{alg:optimistic_frank_wolfe} we write $D_\psi(v, v_{t-1}, \phi_t \in \in \partial \psi(v_{t-1}))$ to mean that the algorithm can use $D_\psi(v, v_{t-1}, \phi_t)$ for any  $\phi_t \in \partial \psi(v_{t-1})$.

Online learning algorithms typically use regularizers $\psi$ that are strongly convex or enjoy any other curvature property such as uniform convexity, in order to obtain a low enough regret. Even though we do not assume strong convexity or any other curvature property of the regularizer $\psi$ (in fact, $\psi$ can be $0$ in the feasible set), we show that we can apply an optimistic approach that leads to the optimal convergence rate of the algorithm in this setting, and we show in \cref{sec:experiments} that empirically outperforms other approaches. 

The starting point of the method, as in \cref{eq:lower_bounds}, is defining a lower bound on the optimal value, which we do by taking inspiration on the anytime online-to-batch conversion of \citet{cutkosky2019anytime}, that connects the regret of online learning algorithms with convergence guarantees of optimization methods. This naturally leads to the definition $x_t = \frac{1}{A_t} \sum_{i=1}^t a_i v_i = \frac{a_t}{A_t} v_t + \frac{A_{t-1}}{A_t} x_{t-1}$ for all $t \geq 1$ in \cref{alg:optimistic_frank_wolfe}.  We are also able to show that at each iteration, the theory allows for making use of a point $y_{t-1}$ with lower function value than the previous $x_{t-1}$, in place of using $x_{t-1}$, and we show that this does not hurt the convergence guarantee, so it allows for heuristics like performing line search over over the segment $x_{t-1}$ and $v_{t}$, or other future heuristics that may be built on top of this algorithm. The definition of $v_t$ comes from the optimistic online learning algorithmic schemes.

In a similar fashion to the one in \cref{eq:def_primal_dual_gap}, using said lower bound we define a primal-dual gap that we denote $G_t^{\operatorname{OP}}$ for both variants of the algorithm. We provide guarantees on this primal-dual gap by the regret of the optimistic procedure, gradient Lipschitzness of the function, along with the loss weighting and compactness of the domain.

We use an optimistic FTRL algorithm (\cref{line:fw_oftrl}) or an optimistic MD algorithm (\cref{line:fw_oftrl}) with \emph{constant step size}, designed to work for subdifferentiable losses. Interesting, while it is well-known that for constant step size and unconstrained problems, FTRL and OMD have the same updates, we note the more general property that in the constrained setting, FTRL is an instance of OMD with subdifferentiable regularizers for a precise choice of subgradients, cf. \cref{remark:equivalence_of_ftrl_and_omd}.

\begin{algorithm}
    \caption{Optimistic Frank-Wolfe algorithms}
    \label{alg:optimistic_frank_wolfe}
   \begin{algorithmic}[1]
       \REQUIRE Convex subdifferentiable regularizer $\psi$ such that $\argmin_x \{\innp{w, x}+\psi(x)\}$ exists for all $w \in \Rd$. A convex function $f$, differentiable and $L$-smooth with respect to a norm $\norm{\cdot}$ in $\operatorname{dom}(\psi)$. Initial point $x_0 = v_0 \in \operatorname{dom}(\psi)$. 
       \STATE $A_{0} \gets 0$; $a_0 \gets 0$
       \vspace{0.2cm}
       \hrule
       \vspace{0.2cm}
       \STATE $g_0 = 0$; \ \ $g_1 = \nabla f(x_0)$
       \FOR{$ t \gets 1 $ \textbf{to} $T$} 
           \STATE $a_t \gets 2t$
           \STATE $A_t \gets A_{t-1} + a_t = \sum_{i=1}^t a_i = t(t+1)$
           \STATE {\color{mygray}$\diamond$ Choose either \ref{line:fw_oftrl} or \ref{line:fw_omd} for the entire run:}
           \STATE $v_t \gets $ point in $\argmin_{v\in\R^d} \{\sum_{i=1}^{t-1} a_i\innp{\nabla f(x_i), v} + a_{t} \innp{g_t, v} + \psi(v)\}$\label{line:fw_oftrl}
           \STATE $v_t \gets $ point in $\argmin_{v\in\R^d} \{a_{t-1}\innp{\nabla f(x_{t-1})-g_{t-1}, v} + a_{t}\innp{g_t, v} + D_\psi(v, v_{t-1}; \phi_t \in \partial\psi(v_{t-1}))\}$\label{line:fw_omd}
           \STATE $y_{t-1} \gets $ point such that $f(y_{t-1}) \leq f(x_{t-1})$
           \STATE $x_{t} \gets \frac{A_{t-1}}{A_t} y_{t-1} + \frac{a_{t}}{A_t} v_t$ \COMMENT{ $= \frac{t-1}{t+1}y_{t-1} + \frac{2}{t+1}v_t$}\label{line:x_in_FW}
           \STATE $g_{t+1} \gets \nabla f(x_t)$
       \ENDFOR
        \STATE \textbf{return} $x_T$.
\end{algorithmic}
\end{algorithm}

One subtlety of $G_t^{\operatorname{OP}}$ is that it is not directly computable as is, since the direction we apply the \LMO{} depends on the hint that we choose. We can compute a simple close bound of it at every iteration, or if we want to compute it after certain number of iterations, we can do it by performing an extra \LMO{}, cf. \cref{remark:computable_optimistic_primal_dual_gap}.

The guarantee we obtain on the algorithm is the following.

\begin{restatable}{theorem}{OptimisticFWGuarantees}\label{thm:optimistic_FW_guarantees}\linktoproof{thm:optimistic_FW_guarantees}
    Let $\X$ be compact and convex, and let $\psi : \X \to \R$ be a closed convex function, subdifferentiable in $\X$.
    
    Let $f$ be convex and $L$-smooth in the set $\X$ of diameter $D$ with respect to a norm $\norm{\cdot}$. The iterates $x_t$ of \cref{alg:optimistic_frank_wolfe} satisfy:
    \[
        f(x_t) - f(x^\ast) \leq G_t^{\operatorname{OP}} \leq \frac{\psi(x^\ast)-\psi(x_1)}{t(t+1)} + \frac{4LD^2}{t+1},
    \] 
    for the variant of \cref{line:fw_oftrl}. For the variant of \cref{line:fw_omd} we obtain the same except that $\psi(x^\ast)-\psi(x_1)$ is substituted by $D_{\psi}(x^\ast, x_0; \phi_0)$, where $\phi_0 \in \partial\psi(x_0)$.
\end{restatable}

We note that the decay of the term involving $\psi$ is just a consequence of our choice of step sizes, which we used for simplicity. It is possible to keep the $\bigo{\frac{LD^2}{t+1}}$ rate and have an arbitrarily fast polinomial-on-$t$ decay on the term involving $\psi$ by making a different choice of step sizes, see \cref{sec:convergence_rate_optimistic_FW}. 

\section{Primal-dual short steps}\label{sec:primal-dual-short-step}

A key step of the analysis in most Frank-Wolfe algorithms for smooth problems, such as in \cref{sec:preliminaries}, consists of using the smoothness inequality in order to guarantee some descent that compensates other per-iteration errors that appear. Several step-size rules have been devised, that may change depending on the specific setting. However, three families of step sizes stand out in almost every setting, which correspond to the ones we discussed in \cref{sec:preliminaries}: (A) open-loop step-sizes, that only depend on the iteration count, (B) short steps, that minimize the upper bound given by the last computed gradient $\nabla f(x_t)$ and the smoothness inequality with respect to some norm, along the segment in between the current point $x_t$ and the computed Frank-Wolfe vertex $v_t \in \argmax_{v\in \X} \{\innp{\nabla f(x_t), v}\}$, and (C) line search in the aforementioned segment to maximize primal progress.

In the sequel, we devise a new class of step sizes, which are a generalization of (B). The idea of (B) is to greedily maximize the \textit{guaranteed} primal progress along the segment in between $x_t$ and $v_t$, where we know we are feasible. The key idea of our new step-size rule consists of taking the primal-dual gap \cref{eq:def_primal_dual_gap} that is defined for the analyses with the structure of \cref{eq:adgt_bound}, and choosing the step size in order to maximize the \textit{guaranteed} progress in terms of this primal-dual gap. 

We also show that our primal-dual gap bound at iteration $t$ is convex with respect to the step size, which implies that one can efficiently do a line search to maximize the primal-dual progress. In order to show the flexibility of this paradigm, we also generalize these ideas to the gradient descent algorithm.

\subsection{Primal-dual short steps for FW algorithms}

Let us consider first the case of \FW{} and FW-HB. In this section, we denote $G_t$, and $x_t$ and $v_t$ the respective primal-dual gaps, and iterates. From the analyses in \cref{sec:review_of_primal_dual_fw}, cf. \eqref{eq:primal_dual_good} or \eqref{eq:primal_dual_good_hb}, we have the following
\begin{align*}
	\begin{aligned}
		A_t& G_t - A_{t-1} G_{t-1} \leq \frac{L a_t^2}{2A_t}\norm{v_{t}-x_t}^2,
	\end{aligned}
\end{align*}
and defining $\gamma_t \defi a_t / A_t$, dividing the above by $A_t$ on both sides, and rearranging gives
\begin{align}\label{eq:isolating_primal_dual_gap}
	\begin{aligned}
		G_t \leq (1- \gamma_t) G_{t-1} + \gamma_t^2 \frac{L}{2}\norm{v_{t}-x_t}^2.
	\end{aligned}
\end{align}
The right-hand side is minimized for the choice 
\begin{align}
    \label{eq:primal-dual-short}
    \gamma_t = \min\left\{1,\frac{G_{t-1}}{L \norm{v_{t}-x_t}^2}\right\},
\end{align}
which is what we refer to as the \emph{primal-dual short step} for these two algorithms, and focuses on maximizing guaranteed progress of the primal-dual gap, as discussed above. We show that this step-size is sound in the sense that we still keep the optimal convergence guarantees of other approaches. 

We also show that we can perform a line search over the primal-dual gap bound obtained before using the smoothness inequality in the analyses. 

\begin{proposition}\label{prop:primal_dual_steps}\linktoproof{prop:primal_dual_steps}
    Let $f$ be convex and differentiable. The \FW{} and \HB{} algorithms satisfy
\begin{align}\label{eq:line_search_expression_primal_dual}
     \begin{aligned}
         G_t &\leq (1-\gamma) G_{t-1}  - \gamma \innp{\nabla f(x_t), v_t - x_t} - f(x_t)\\
         & + f((1-\gamma) x_t + \gamma v_t), \quad \quad \text{ for all } \gamma \in [0, 1], t>1
     \end{aligned}
\end{align}
    and the RHS is convex on $\gamma$. 
    If further $f$ is $L$-smooth w.r.t. a norm $\norm{\cdot}$, and $D \defi \max_{x, y\in\X}\norm{x-y}$, then using the step \cref{eq:primal-dual-short} or line search on the RHS of \cref{eq:line_search_expression_primal_dual} above, for either algorithm, we obtain:
    \[
        G_t \leq \frac{4LD^2}{t+2}.
    \] 
\end{proposition}
The convergence above for the line search is derived from the one for the primal-dual short step, since the right hand side of \cref{eq:primal-dual-short} upper bounds the one of \cref{eq:line_search_expression_primal_dual}.

    We note that, naturally, if we define the gap as $G_t \defi f(x_{t+1}) - f(x^\ast)$, which is the primal gap, then analyzing the algorithm with the strategy in \cref{eq:adgt_bound}, yields that the primal-dual short steps become regular short steps. In that case, even though we do not know the value $f(x^\ast)$ in the gap definition, the step can still be defined without this knowledge, cf. \cref{remark:pd_short_steps_are_a_generalization}.

\subsection{Primal-dual short steps for Gradient Descent}

\label{sec:primal-dual-short-step-gd}
We now extend the analysis of our primal-dual short steps to the gradient descent algorithm (\newtarget{def:acronym_gradient_descent}{\GD{}}) whose updates are given by $x_{t+1} \gets x_t - a_t \nabla f(x_t)$. We use the Euclidean norm in this section for the problem $\min_{x\in \Rd} f(x)$, where $f$ is convex and $L$-smooth. 

One possible analysis of \GD{} relies on defining a primal-dual gap similarly to the one in \cref{sec:FW_with_regularization}, and then show $A_t G_t \leq A_{t-1}G_{t-1}$ for $a_t = 1 / L$ and $A_1 G_1 = \frac{1}{2}\norm{x_1 - x^\ast}_2^2$, for an initial point $x_1$ and a minimizer $x^\ast$. We provide a sketch of the steps that we take and leave the details to \cref{sec:primal_dual_steps_gd}. The lower bound on $f(x^\ast)$ that we use is
\begin{align*}
     \begin{aligned}
         A_t f(x^\ast) &\geq  \sum_{i=1}^t a_i f(x_i) + \sum_{i=1}^t a_i \innp{\nabla f(x_i), x_{t+1} - x_i} \\
         &\quad + \frac{1}{2}\norm{x_{t+1}-x_1}_2^2 - D^2 \defi A_t L_t, 
     \end{aligned}
    \end{align*}
    where $D$ is an upper bound on the initial distance to a minimizer $\norm{x^\ast-x_1}_2$. Defining the gap as $G_t = f(x_{t+1}) - L_t$, we arrive to
\[
    G_t{\leq}\frac{A_{t-1}}{A_t}G_{t-1} +\norm{\nabla f(x_t)}_2^2 \left(- \frac{a_t A_{t-1}}{A_t} + \frac{a_t^2L}{2} - \frac{a_t^2}{2A_t}\right),
\] 
for $t > 1$ and a similar expression for $t=1$. Recall that $A_t \defi \sum_{i=1}^t a_t$. So at this stage, one can optimize $a_t$ in order to minimize the right hand side, by solving a simple cubic equation, which is what we term the primal-dual short steps for \GD{}. 

We show that the optimal value satisfies $a_t \geq \frac{1}{2L}$, which ultimately yields to a convergence rate no slower than $G_t \leq \frac{LD^2}{t}$, as we formalize in the following. We also show, as above, that we can perform a line search for minimizing a better bound on the primal-dual gap, yielding no worse convergence rates. Recall that $f(x_{t+1}) - f(x^\ast){\,\leq\,}G_t$.

\begin{proposition}\label{prop:gd_primal_dual_steps}\linktoproof{prop:gd_primal_dual_steps}
    Let $f$ be convex and differentiable. The primal-dual gap of \GD{} is bounded by \cref{eq:without_using_smoothness} which is convex on the step $a_t$. If further $f$ is $L$-smooth with respect to $\norm{\cdot}_2$, \GD{} using the primal-dual short step-size or line search on the aforementioned bound then the step size satisfies $a_t \geq \frac{1}{2L}$ and we have 
\[
G_t \leq \frac{LD^2}{t}.
\] 
\end{proposition}
Interestingly, a notable difference between the line search on these primal-dual bounds in \GD{} and \FW{} algorithms is that convexity was shown in \cref{prop:primal_dual_steps} to hold with respect to the parameter $\gamma$ which corresponds to $\frac{a_t}{A_t}$ whereas for \GD{} in \cref{prop:gd_primal_dual_steps}, the convexity is shown to hold for the parameter $a_t$. 

\section{Experiments}\label{sec:experiments}

\begin{figure*}[ht!]
    \begin{center}
        \includegraphics[width=0.80\textwidth]{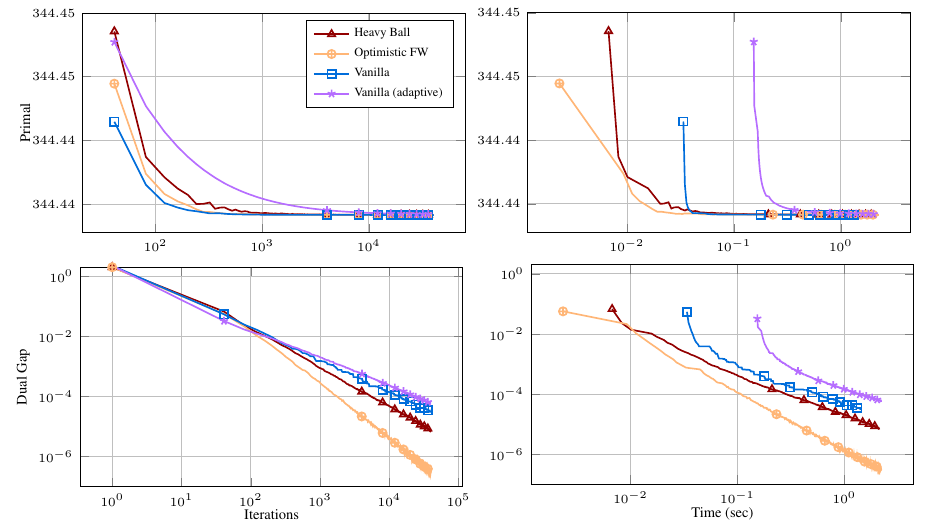}
    \end{center}
    \caption{\label{fig:sq_optimistic_fw_oftrl_prob_simplex} Comparison over the probability simplex of dimension $n=1000$ with objective $f(x) = \norm{x - x_0}_2^2$, where $x_0$ is a random point outside the probability simplex. We can see that the optimistic variant converges faster than the other variants both in iterations and time. Note that we cut datapoints with excessively large primal/dual values, which leads to apparent different starting points in the graphs.}
\end{figure*}

In the following we provide some preliminary experiments demonstrating the good performance of the optimistic variant. All experiments were performed in Julia based on the \texttt{FrankWolfe.jl} package run on a MacBook Pro with an Apple M1 chip with Julia 1.11.1. The code will be made publicly available upon publication. 

In our experiments we compare the optimistic \FW{} variant (with OFTRL) with the heavy ball variant, and the vanilla \FW{} algorithm. For the \FW{} variant, apart from the agnostic step size rule $\gamma_t \doteq \frac{2}{t+2}$ we also consider the adaptive line search (indicated by \texttt{adaptive} in the plots) from \citep{P2023}, a numerically improved variant of \citep{pedregosa2020linearly}, which is the default line search in the \texttt{FrankWolfe.jl} package. As shown in \citep{guelat1986some}, line search is subject to a lower bound that does not hold for, e.g., open-loop step-sizes and in particular the heavy ball variant and the optimistic variants. In fact, there are cases where the line search variant is slower than the vanilla variant as asymptotically shown in \citep{bach2021effectiveness} and later in \citep{WKP2022,WPP2023} for the non-asymptotic case; we observe a similar behavior in our experiments as they satisfy the conditions of \citep{WKP2022,WPP2023}; see also \citep{KAP2021}. 

Our numerical experiments are over various polytopes $P$, so that in particular $\psi(x) \defi \indicator{P}(x)$ is the indicator function of the respective polytope. In the following we will use for the dual gap the definition of $G_t \defi f(x_t) - L_t$, where the lower bound function is defined as 
\begin{align*}
    \begin{aligned}
        f(x^\ast) &\geq \sum_{i=0}^t \frac{a_i}{A_t} f(x_i) + \sum_{i=0}^t \frac{a_i}{A_t} \innp{\nabla f(x_i), x^\ast - x_i}  \\
        &\geq  \sum_{i=0}^t \frac{a_i}{A_t} f(x_i) + \sum_{i=0}^t \frac{a_i}{A_t} \innp{\nabla f(x_i), v_i - x_i}  \defi L_t.
    \end{aligned}
   \end{align*}
In the case of the heavy ball variant from \cref{sec:heavyBall} we set $v_i = v_t$, for all $i \leq t$, and 
$$v_t \in \argmin_{v \in P} \sum_{i=0}^t \frac{a_i}{A_t} f(x_i) + \sum_{i=0}^t \frac{a_i}{A_t} \innp{\nabla f(x_i), v_t - x_i},$$
and we refer to this lower bound as $L_t^{\operatorname{HB}}$ if not clear from the context.
Clearly this is a lower bound and in fact this is precisely the heavy ball vertex that is computed in the heavy ball algorithm. It is important to note that this lower bound and the associated dual gap is valid for \emph{any algorithm} as it is simply a convex combination of linear lower bound functions minimized over $P$. For comparability and due to its natural form and being valid for any algorithm, we use this definition of the gap also for reporting the results of the optimistic variant in our experiments.

In the case of the vanilla \FW{} variant (either with open loop step sizes or with line search) the $v_i$ in the expression above are chosen to be the \FW{} vertices, i.e., 
$$v_i \in \argmin_{v \in P} f(x_i) + \innp{\nabla f(x_i), v - x_i},$$
minimizing each summand separately. This bound is thus separable in contrast to the heavy ball one as we will discuss further below in \cref{rem:separable_gap} and we can use
\begin{align*}
    \begin{aligned}
        f(x^\ast) & \geq \max_{0 \leq i \leq t} f(x_i) + \innp{\nabla f(x_i), v_i - x_i},
    \end{aligned}
   \end{align*}
as lower bound so that the gap becomes the running minimum of the \FW{} gaps across the iterations, which we refer to as $L_t^{\operatorname{FW}}$.

\begin{remark}[Strength of lower bounds]
    \label{rem:separable_gap}
    Suppose that $g(x_t)$ denotes a generic \emph{gap function} that bounds the primal gap at $x_t$, i.e., $f(x_t) - f(x^*) \leq g(x_t)$. We will choose the specific gap function later depending on the context. Observe that the gap function immediately gives a lower bound for $f(x^*)$ simply by rewriting as
    $$
    f(x^*) \geq f(x_t) - g(x_t). 
    $$
    Then in line with the above if we make the choice for the lower bound $L_t$ as
    $$A_t f(x^*) \geq \sum_{\ell = 0}^t a_\ell f(x_\ell) - \sum_{\ell = 0}^t a_\ell g(x_\ell) = A_t L_t,$$
    then in this case, the lower bound $L_t$ cannot be stronger than taking the best lower bound observed so far, since
    \begin{align*}
    f(x^*) & \geq \max_{\ell = 0,\dots, t} f(x_\ell) - g(x_\ell)\\
        &\geq \frac{1}{A_t}\sum_{\ell = 0}^t a_\ell (f(x_\ell) -  g(x_\ell)), 
    \end{align*}
    This is the case because the lower bound is a convex combination of lower bound terms from individual iterations. 

\begin{figure*}[ht!]
    \begin{center}
        \includegraphics[width=0.80\textwidth]{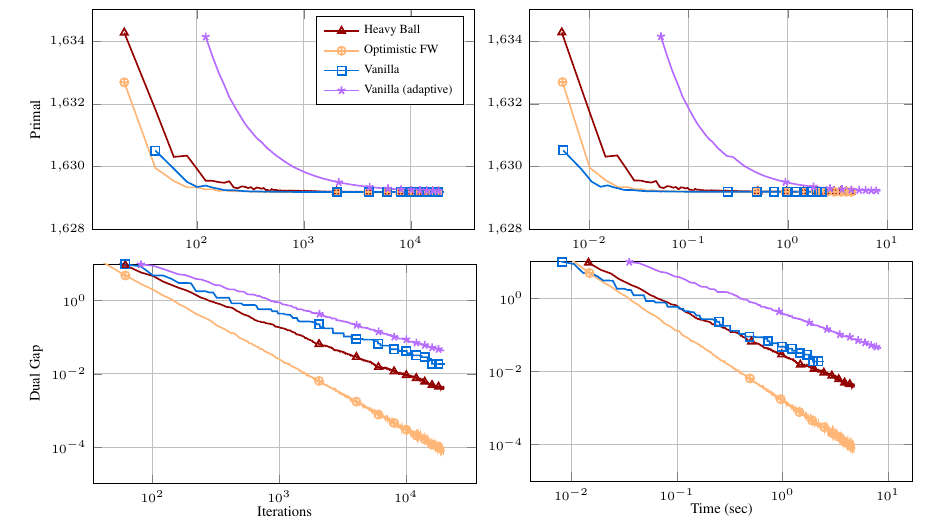}
    \end{center}
    \caption{\label{fig:sq_optimistic_fw_oftrl_prob_simplex_2} \label{fig:non-sq_optimistic_fw_oftrl_k-sparse} Comparison over $k$-sparse polytope of dimension $n=100$ and $k = 10$ with objective $f(x) = \norm{Ax - b}_2^2$, where $A$ and $b$ are random. We can see that the optimistic variant also converges faster than the other variants both in iterations and time. Here we also cut datapoints with excessively large primal/dual values, which leads to apparent different starting points in the graphs.}
\end{figure*}
    
    As mentioned, this is different, e.g., for the heavy ball lower bound function, which does not decompose in individual iterations as the Frank-Wolfe vertex is computed for the cumulative function across rounds. The reason why the primal-dual gap for \HB{} can be significantly better (as in: lower) than the Frank-Wolfe gap is illustrated in the following example: The \FW{} algorithm, commonly used with polytopal constraints, sometimes suffers from the so-called zigzag problem \citep{wolfe1970convergence,guelat1986some} (see also \citet{CGFWSurvey2022} for an in-depth discussion), that usually is due to a minimizer being in the relative interior of a face while the points $v_t$, being the result of an \LMO{}, are chosen as vertices of the polytope. In this scenario, it is possible that $\nabla f(x_t)$ is aligned with the direction $v_t - x_t$ for all $t$, while a convex combination of gradients is close to being perpendicular to the optimal face making the lower bound $L_t^{\operatorname{HB}}$ be closer to the optimal value.
\end{remark}

In the plots we show the primal value and the dual gap vs. the number of iterations and time. We use log-log plots so that the slope is equal to the polynomial order of convergence. As our variants compute one gradient per iteration, the number of iterations is equal to the number of gradient evaluations. Throughout our experiments we see that the optimistic variant significantly outperforms the heavyball variant, the vanilla \FW{}, and often, although not always, also the \FW{} variant with the adaptive line search strategy from \citep{P2023} in the order of convergence.

We also run experiments for the primal-dual short step from \cref{sec:primal-dual-short-step}. However, we found that it behaved similarly than the vanilla \FW{} variant with standard short-steps (or line search), without outperforming this already good heuristic; see \cref{fig:pdss-sq-prob-simplex} in \cref{sec:experiments_appendix}. The reason seems to be that because the primal-dual gap measure is better, the step size is actually often \emph{smaller}, leading to no faster convergence; this in fact is reminiscent of the lower bound in \citep{guelat1986some} and not completely surprising. 

Finally, we also tested whether optimism is really the main source of improvement or whether most of it is explained by using the heavy ball lower bound which shows to be stronger than the vanilla \FW{} lower bound. In fact, we show in \cref{fig:sfw-nsq-prob-simplex} in \cref{sec:experiments_appendix}, that indeed optimism is what is making the algorithm be faster. In fact the heavy ball lower bound applied to the vanilla \FW{} variant is weaker that the normal \FW{} gap in our tests. This strongly suggests that it is really the trajectory of the optimistic variant that is better.

We tested various combinations of functions and feasible regions. For the functions, we used (a) a simple quadratic function of the form $f(x) = \norm{x-x_0}^2$ with $x_0$ being drawn randomly. This is basically euclidean projection onto the feasible region. (b) a more complex, ill-conditioned quadratic function of the form $f(x) = x^T Q x + c^T x$, where the data was randomly generated so that $Q$ is positive definite and $c$ is a random vector. In terms of the feasible regions, we considered (i) the probability simplex $P = \{x \in \R^d \mid \sum_i x_i = 1, x \geq 0\}$, (ii) and the more complex $k$-sparse polytope $P$ which is the convex hull of all $0/1$-vectors with $k$ non-zero entries. We detail the instance parameters in the captions of the figures.

In our preliminary experiments we report results for quadratics, i.e., regression problems. The reason for this is to allow us to build specific setups where the different convergence behaviors of \FW{} variants are known to emerge. We also performed preliminary experiments over portfolio optimization instances (here the objective is the negative log-likelihood of the portfolio returns) and also Optimal Experiment Design instances \citep{hendrych2023solving} (here the objective arises via matrix means) and found that the optimistic variant also outperforms the other variants similarly. We intend to perform at-scale experiements with an optimized implementation within the \texttt{FrankWolfe.jl} package.

\section*{Acknowledgements}
 Research reported in this paper was partially supported through the Research Campus Modal funded by the German Federal Ministry of Education and Research (fund numbers 05M14ZAM,05M20ZBM) and the Deutsche Forschungsgemeinschaft (DFG) through the DFG Cluster of Excellence MATH+. David Martínez-Rubio was partially funded by the project IDEA-CM (TEC-2024/COM-89).

\bibliography{refs}
\bibliographystyle{icml2025}

\clearpage

\appendix

\onecolumn

\section{Overview and some generalizations of Frank-Wolfe's primal-dual analyses}\label{sec:review_of_primal_dual_fw}

As discussed in the introduction, the analysis of Frank-Wolfe algorithms on compact convex sets can be generalized to the problem $f + \psi$ where we assume access to a gradient oracle for $f$ and an that solves the global optimization of the function $\argmin_{x\in \Rd}\innp{w, x} + \psi(x)$ for every $x$, as originally proposed in \citep{nesterov2018complexity}. Here, we provide a short proof of this fact, with a slightly better constant, which in particular proves the claims in \cref{sec:preliminaries} when $\psi$.

Assume $f: \R^n \to \R$ is convex, $L$-smooth, and differentiable in an open set containing $\X \defi \dom(\psi)$ and $\psi: \R^n \to R$ is a convex, proper, lower semi-continuous function. We assume that for every $u \in \R^n$, there exists $\argmin_x\{\innp{u, x} + \psi(x)\}$. Finally, the convergence after $t$ steps will depend on $D_t \defi \max_{\ell=0, \dots, t}\{\norm{v_\ell - x_\ell}\}$ for the points $x_\ell, v_\ell$ with $\ell=0, \dots, t$ as defined in \cref{alg:fw} up to iteration $t$. Similar to the original \FW{} algorithm, we can instead substitute $D_t$ by an upper bound of such quantity, e.g., $\diam(\dom(\psi))$, assuming its value is finite. Recall that we define positive weights $a_\ell > 0$ for all $\ell \geq 0$,  to be determined later, and $A_t \defi \sum_{i=0}^{t} a_i$. Thus, $A_{-1} = 0$.

\begin{algorithm}[H]
    \caption{\label{alg:fw_generalized} Generalized Frank-Wolfe and Heavy-Ball \FW{} algorithms }
    \begin{algorithmic}[1]
        {\footnotesize
        \REQUIRE Functions $f$ and $\psi$, initial point $x_0$. Weights $a_t, \gamma_t$.
            \FOR{$t=0$ \textbf{to} $T$}
            \STATE $a_t \gets 2t + 2$, $A_t \gets \sum_{i=0}^t a_i = (t+1)(t+2)$
           \STATE {\color{mygray}$\diamond$ Choose either \ref{line:heavy_ball_fw} (generalized \newlinkcolor{def:acronym_heavy_ball_frank_wolfe}{HB-FW}{mygray}) or \ref{line:vanilla_fw} (generalized \newlinkcolor{def:acronym_frank_wolfe}{FW}{mygray}) for the entire run:}
            \STATE $v_t \leftarrow \argmin\limits_{v \in \Rd} \Big\{ \sum_{i=0}^{t} a_i\langle \nabla f(x_i), v \rangle + A_t \psi(v) \Big\}$ \label{line:heavy_ball_fw_generalized}
            \STATE $v_t \leftarrow \argmin\limits_{v \in \Rd} \left\{ \langle \nabla f(x_t), v \rangle + \psi(v) \right\}$\label{line:vanilla_fw_generalized}
            \STATE $x_{t+1} \leftarrow \frac{A_{t-1}}{A_t} x_t + \frac{a_t}{A_T} v_t$
            \ENDFOR}
    \end{algorithmic}
\end{algorithm}

We define the following lower bound on $f(x^\ast) + \psi(x^\ast)$, which is a generalization of \cref{eq:lower_bounds}, where $x^\ast$ here is defined as a minimizer in $\argmin_{x\in\R^n} \{ f(x) + \psi(x) \}$:
\begin{align*}
    \begin{aligned}
        & \ A_t (f(x^\ast) + \psi(x^\ast)) \circled{1}[\geq] \sum_{\ell=0}^t a_\ell f(x_\ell) + \sum_{\ell=0}^t a_\ell \innp{\nabla f(x_\ell), x^\ast - x_\ell} + \sum_{\ell=0}^t a_\ell \psi(x^\ast)  \\
        &\circled{2}[\geq]  \sum_{\ell=0}^t a_\ell f(x_\ell) + \sum_{\ell=0}^t a_\ell \innp{\nabla f(x_\ell), v_\ell - x_\ell} + \sum_{\ell=0}^t a_\ell\psi(v_\ell)  \\
        & \defi A_t L_t, 
    \end{aligned}
\end{align*}
where we applied convexity in $\circled{1}$, and for $\circled{2}$, we applied the definition of $v_\ell$ in the algorithm, which is why we define this point. 
    Define the primal-dual gap  
    \begin{equation}
        G_t \defi f(x_{t+1}) - L_t
    \end{equation}
    and note that upper bound $f(x_{t+1})$ is one step ahead, which helps the analysis. We obtain the following, for $t \geq 0$: 
\begin{align}\label{eq:primal_dual_good}
    \begin{aligned}
        &A_t G_t - A_{t-1} G_{t-1} = A_{t} f(x_{t+1}) - A_{t-1}f(x_{t}) + A_{t}\psi(x_{t+1}) - A_{t-1}\psi(x_{t}) \\
        & - \left( a_t f(x_{t}) + \Ccancel[red]{\sum_{\ell=0}^{t-1} a_\ell f(x_\ell)} + a_t\innp{\nabla f(x_t), v_t - x_t} + a_t\psi(v_t) + \Ccancel[blue]{\sum_{\ell=0}^{t-1} a_\ell(\innp{\nabla f(x_\ell), v_\ell - x_\ell} + \psi(v_\ell)) } \right)  \\
        & + \left( \Ccancel[red]{\sum_{\ell=0}^{t-1} a_\ell f(x_\ell)} + \Ccancel[blue]{\sum_{\ell=0}^{t-1}  a_\ell(\innp{\nabla f(x_\ell), v_\ell - x_\ell} + \psi(v_\ell)) } \right) \\
        &\circled{1}[\leq] \innp{\nabla f(x_t), A_{t}(x_{t+1} - x_{t}) - a_t (v_t - x_t)} + \frac{ L A_t}{2}\norm{x_{t+1}-x_t}^2 \\
        &\circled{2}[=] \frac{L a_t^2}{2A_t}\norm{v_{t}-x_t}^2  \circled{3}[\leq] \frac{LD_t^2 a_t^2}{2A_{t}}  \defi E_t.
    \end{aligned}
\end{align}
In $\circled{1}$ we grouped terms to get $A_t(f(x_{t+1})-f(x_t))$ and applied smoothness to this term. We also used the definition of $x_{t+1}$ which along with the convexity of $\psi$, implies $A_t \psi(x_{t+1}) \leq A_{t-1}\psi(x_t) + a_t \psi(v_t)$. In $\circled{2}$ we used twice the definition of $x_{t+1}$ which implies $A_{t} (x_{t+1}-x_t) = a_t (v_{t}-x_t)$. In $\circled{3}$, we bounded the distance of points by the quantity $D_t$. 

Now with the choice $a_t = 2t+2$ and $A_t = \sum_{i=0}^t a_i = (t+1)(t+2)$ that
\begin{align}\label{eq:primal_dual_good_rate}
    \begin{aligned}
        f(x_{t+1})+\psi(x_{t+1}) -(f(x^\ast)+\psi(x^\ast)) \leq G_t \leq  \frac{1}{A_t}\sum_{i=0}^t \frac{LD_t^2 a_i^2}{2A_{i}} = \frac{1}{(t+1)(t+2)}\sum_{i=0}^t \frac{2LD_t^2(i+1)}{i+2} < \frac{2 LD_t^2}{t+2}.
\end{aligned}
\end{align}

Recall that the classical Frank-Wolfe algorithm corresponds to the algorithm and analysis presented above but for $\psi$ being the indicator function $\indicator{\X}$ of a compact convex set $\X$. In such a case, we have the following.

\begin{remark}[Alternative step-sizes in Generalized Frank-Wolfe]
    We note that step-size strategies like the ones described in \cref{eq:non_open_loop_step_sizes} also work in the general case $f+\psi$, by finding 
\[
    \argmin_{x \in\conv{\hat{x}_t, v_t}} \left\{ f(\hat{x}_t) + \innp{\nabla f(\hat{x}_t), x-\hat{x}_t} + \frac{L}{2}\norm{x -\hat{x}_t}^2+\psi(x)\right\} \text{ or } \argmin_{x \in\conv{\hat{x}_t, v_t}} \{f(x) + \psi(x)\},
\] 
respectively.
\end{remark}

\subsection{The Heavy Ball Frank-Wolfe algorithm}
\label{sec:heavyBall}

We will now consider a variant of the Frank-Wolfe algorithm that uses a heavy ball step, similar to the one from \cite{li2021heavy}, however allowing for more flexibility in the choice of the step size parameters and an additional term $\psi(v)$.  As before, let $a_i > 0$ to be determined later and define $A_t = \sum_{i=0}^{t} a_i$.
Let $v_t \defiin \argmin_{v\in\X} \left\{ \innp{\sum_{i=0}^t a_i\nabla f(x_i), v} + A_t \psi(v) \right\}$ and let $x_{t+1} \defi \frac{A_{t-1}}{A_t}x_{t} + \frac{a_t}{A_t} v_t$ be defined as a convex combination of $x_{t}$ and $v_{t}$ (and since $A_0 = a_0$, we have $x_1 =v_0$). We define the following lower bound on $f(x^\ast)$, where $x^\ast$ is defined as a minimizer in $\argmin_{x\in\X} \{ f(x) + \psi(x) \}$:
    \begin{align*}
     \begin{aligned}
         A_t ( f(x^\ast) + \psi(x^\ast) ) &\circled{1}[\geq] \sum_{i=0}^t a_i f(x_i) + \sum_{i=0}^t a_i \innp{\nabla f(x_i), x^\ast - x_i}  + A_t \psi(x^\ast) \\
         &\circled{2}[\geq]  \sum_{i=0}^t a_i f(x_i) + \sum_{i=0}^t a_i \innp{\nabla f(x_i), v_t - x_i} + A_t \psi(v_t)  \defi A_t L_t  \\
     \end{aligned}
    \end{align*}
    where we applied convexity in $\circled{1}$, and the definition of $v_t$ in $\circled{2}$. 
    As before, we define the primal-dual gap  as
    \begin{equation}
        G_t \defi f(x_{t+1}) - L_t.
    \end{equation}
    We obtain the following, for $t \geq 0$: 
    \begin{align}\label{eq:primal_dual_good_hb}
     \begin{aligned}
         A_t& G_t - A_{t-1} G_{t-1} = A_{t} f(x_{t+1}) - A_{t-1}f(x_{t}) + A_t \psi(x_{t+1}) - A_{t-1} \psi(x_{t}) \\
         & \ \ - \left( a_t f(x_{t}) + \Ccancel[red]{\sum_{i=0}^{t-1} a_i f(x_i)} + \sum_{i=0}^{t-1} a_i\innp{\nabla f(x_i), v_t - x_i} + A_{t-1}\psi(v_t)  \right) - a_t\innp{\nabla f(x_t), v_t - x_t} - a_t \psi(v_t)   \\
         & \ \ + \left( \Ccancel[red]{\sum_{i=0}^{t-1} a_i f(x_i)} + \sum_{i=0}^{t-1} a_i\innp{\nabla f(x_i), v_{t-1} - x_i} + A_{t-1}\psi(v_{t-1})\right) \\
         &\circled{1}[\leq] \innp{\nabla f(x_t), A_{t}(x_{t+1} - x_{t}) - a_t (v_t - x_t)}  + \frac{ L A_t}{2}\norm{x_{t+1}-x_t}^2 \\
         &\circled{2}[\leq] \frac{L a_t^2}{2A_t}\norm{v_{t}-x_t}^2  \circled{3}[\leq] \frac{LD_t^2 a_t^2}{2A_{t}}  \defi E_t.
     \end{aligned}
    \end{align}
    In $\circled{1}$ we grouped terms to get $A_t(f(x_{t+1})-f(x_t))$ and applied smoothness to this term. We also dropped the rest of the terms in the parentheses by optimality of of $v_{t-1}$ and dropped the three other terms depending on $\psi$ by its convexity and $A_t x_{t+1} = A_{t-1} x_t + a_t v_t$. In $\circled{2}$ we used the definition of $x_{t+1}$, which is the point that minimizes the right hand side of the smoothness inequality that we applied and so substituting $x_{t+1}$ by $\frac{A_{t-1}}{A_t}x_{t} + \frac{a_t}{A_t} v_t$ leads to something greater. We simplified some terms. In $\circled{3}$, we bounded the distance of points by the diameter of $\X$. Concluding is now the same as for \FW{}.

\subsection{Generalized Frank-Wolfe algorithm revisited}\label{sec:FW_with_regularization}

In \eqref{eq:primal_dual_good} we have obtained a convergence rate for $\min_x f(x) + \psi(x)$. Alternatively, if we are interested in optimizing $f$, we can use a regularizer $\psi$ and conclude convergence with a very similar analysis. As a consequence, $\psi(x^\ast)$ appears in the rates. If $\psi$ is an indicator function of a set this is exactly the heavy ball algorithm from above but for a general $\psi$, it is different. 

The lower bound that we define on $f(x^\ast)$ with $x^\ast \in \argmin_{x\in\X} f(x)$, and $\psi$ as before, is as follows:
    \begin{align}\label{eq:lower_bound_in_FW_heavy_ball_w_regularizer}
     \begin{aligned}
         A_t f(x^\ast)  &\circled{1}[\geq] \sum_{i=0}^t a_i f(x_i) + \sum_{i=0}^t a_i \innp{\nabla f(x_i), x^\ast - x_i} \pm  \psi(x^\ast)  \\
         &\geq  \sum_{i=0}^t a_i f(x_i) + \sum_{i=0}^t a_i \innp{\nabla f(x_i), v_t - x_i} + \psi(v_t) - \psi(x^\ast)  \defi A_t L_t,
     \end{aligned}
    \end{align}
    where $v_t \defiin \argmin_{v\in\R^n}\{\sum_{i=0}^t a_i \innp{\nabla f(x_i), v} + \psi(v)\}$. Similarly to the previous section, we assume access to an oracle that returns one such $v_t$, which is assumed to exist. We used convexity in $\circled{1}$ and we also added and subtracted $\psi$ in order to compute a lower bound that allows us to reduce the gap without knowing $x^\ast$. Finally, defining the gap $G_t \defi f(x_{t+1}) - L_t$, we have, for all $t\geq 0$
\begin{align}
     \begin{aligned}
         A_t& G_t - A_{t-1} G_{t-1} - \ones_{\{t=0\}}\psi(x^\ast)  = A_{t} f(x_{t+1}) - A_{t-1}f(x_{t})  \\
         & \ \ - \left( a_t f(x_{t}) + \Ccancel[red]{\sum_{i=0}^{t-1} a_i f(x_i)} + a_t\innp{\nabla f(x_t), v_t - x_t} + \sum_{i=0}^{t-1} a_i\innp{\nabla f(x_i), v_t - x_i} + \psi(v_t)  \right)  \\
         & \ \ + \left( \Ccancel[red]{\sum_{i=0}^{t-1} a_i f(x_i)} + \sum_{i=0}^{t-1}  a_i\innp{\nabla f(x_i), v_{t-1} - x_i} + \psi(v_{t-1})  \right) \\
         &\circled{1}[\leq] \innp{\nabla f(x_t), A_{t}(x_{t+1} - x_{t}) - a_t (v_t - x_t)} + \frac{ L A_t}{2}\norm{x_{t+1}-x_t}^2 \\
         &\circled{2}[=] \frac{L a_t^2}{2A_t}\norm{v_{t}-x_t}^2 \circled{3}[\leq] \frac{LD^2 a_t^2}{2A_{t}}  \defi E_t.
     \end{aligned}
    \end{align}
    Step $\circled{1}$ uses the optimality of $v_{t-1}$ to bound some terms by $0$, Steps $\circled{2}$ and $\circled{3}$ are identical to the previous analysis.
Now to conclude we choose $a_t = 2t+2$ as before, and so we have $A_t = \sum_{i=0}^t a_i =  (t+1)(t+2)$. But we also have to take into account that we have the extra term $\psi(x^\ast)$ when bounding $A_0 G_0$:
    \begin{align*}
     \begin{aligned}
         f(x_{t+1}) & - f(x^\ast) \leq \frac{1}{A_t}\left(\psi(x^\ast) + \sum_{i=0}^t E_i \right) = \frac{\psi(x^\ast)}{A_t} + \frac{1}{A_t}\sum_{i=0}^t \frac{LD^2 a_i^2}{2A_{i}} \\
         &= \frac{\psi(x^\ast)}{A_t} + \frac{1}{(t+1)(t+2)}\sum_{i=0}^t\frac{2LD^2(i+1)}{(i+2)} < \frac{\psi(x^\ast)}{(t+1)(t+2)} + \frac{2 LD^2}{t+2}.
     \end{aligned}
    \end{align*}

    \begin{remark}[Arbitrary fast rate for $\psi(x^\ast)$] In fact, the part of the rate involving $\psi(x^\ast)$ was arbitrary and made for simplicity. Indeed, the intuition is that in the lower bound \cref{eq:lower_bound_in_FW_heavy_ball_w_regularizer} we are adding a regularizer that is in fact $\psi(x^\ast)/A_t$ (since the whole inequality is multiplied by $A_t$ in particular $f(x^\ast)$), so the larger we make $A_t$ be, the smaller the regularizer we are adding is, and the faster that term will go to $0$. In algebra, if we set $a_i = \Theta(i^{k})$, then we have that $\sum_{i=0}^t a_i^2/A_i = \bigtheta{\sum_{i=0}^t i^{2k-(k+1)}}  = \bigtheta{t^k}$ and $A_t = \bigtheta{t^{k+1}}$ so in any case the part without regularizer is
        \[
            LD^2\frac{1}{A_t}\sum_{i=0}^t \frac{a_i^2}{A_i} = \bigthetal{\frac{LD^2}{t}},
        \] 
        but now $A_t$ grows at any polynomial rate that we want which we can use to decrease $\psi(x^\ast)/A_t$ fast. If we are adding a $\psi$ ourselves on top of an indicator function because we want some properties to hold, most of the time we would like to keep this term being of the same order as the current gap so its contribution to the total rate of convergence is a multiplicative constant.

        This reasoning also applies when finishing the analysis of \cref{thm:optimistic_FW_guarantees}.
    \end{remark}

\section{Proofs for Optimistc FW algorithm}\label{sec:optimistic_proofs}

We start by defining the lower bound that we use on $f(x^\ast)$ for a family of algorithms, using which we define a primal-dual gap that allows to show convergence of our optimistic Frank-Wolfe algorithm. 
This proof is inspired by the anytime online-to-batch reduction  of an optimization problem into an online learning one \citep{cutkosky2019anytime}, which is a generalization and independent work with respect to \citep{nesterov2015quasi}. An interesting modification of the technique that we make is that we allow for selecting a point $y_{t-1}$ such that $f(y_{t-1}) \leq f(x_{t-1})$ before computing the coupling $x_t$. This allows to incorporate heuristics without degrading the convergence rate, such as performing a line search between $x_{t-1}$ and $v_t$. We used this modified reduction in order to provide a primal-dual gap, which is computable if the regret of the corresponding online learning problem is computable. If we instead have a computable upper bound $\widehat{R}_t(x^\ast)$ on the regret $R_t(x^\ast)$ we can correspondingly define and compute a primal-dual gap based on this bound. This will be the case when we instantiate the framework with our optimistic algorithm.

Let $a_t > 0$ for $t \geq 1$ and define $A_t \defi \sum_{i=1}^t a_i$. Given some points $\{v_i\}_{i\geq 1}$, and given $x_0$, we define $x_t \defi \frac{A_{t-1}}{A_t} y_{t-1} + \frac{a_t}{A_t} v_t$, for $t \geq 1$, where $y_{t}$ is any point such that $f(y_{t}) \leq f(x_{t})$. Then, we have for all $t \geq 1$ and $k  = 0, 1, \dots, t$:
\begin{align}\label{eq:lower_bound_generic_anytime_online_to_batch}
    \begin{aligned}
        A_t f(x^\ast) &\circled{1}[\geq] \sum_{i=1}^t a_i f(x_i) +  \sum_{i=1}^t a_i\innp{\nabla f(x_i), x^\ast - x_i} \\
        &\circled{2}[=] \Phi_k \defi A_k f(x_k) + \sum_{i=1}^k a_i \innp{\nabla f(x_i), x^\ast - v_i} + \sum_{i=k+1}^t a_i f(x_i) + \sum_{i=k+1}^t a_i \innp{\nabla  f(x_i), x^\ast - x_i}\\
        & \quad + \sum_{i=0}^{k-1} \Big( A_i D_f(y_i, x_{i+1}) + A_i (f(x_{i}) - f(y_{i})) \Big) \\
        &\circled{3}[=] A_t f(x_t) + \sum_{i=1}^t a_i \innp{\nabla f(x_i), x^\ast - v_i}  + \sum_{i=0}^{t-1} \Big( A_i D_f(y_i, x_{i+1}) + A_i (f(x_{i}) - f(y_{i})) \Big) \\
        &\circled{4}[\geq] A_t f(x_t) - \widehat{R}_t(x^\ast) \defi A_t L_t 
    \end{aligned}
\end{align}
where $\circled{1}$ uses convexity of $f$. We have that $\circled{2}$ and $\circled{3}$ are due to the terms on the sides being $\Phi_0$ and $\Phi_t$ and to $\Phi_k = \Phi_{k+1}$, which holds, since canceling terms such equality is equivalent to $\circled{5}$ below:
\begin{align*}
    \begin{aligned}
        A_k f(x_k) &\circled{5}[=] a_{k+1} \innp{\nabla f(x_{k+1}), x_{k+1} - v_{k+1}} + A_k f(x_{k+1}) + A_k D_f(y_k, x_{k+1}) + A_k (f(x_k) - f(y_k)) \\
        &\circled{6}[=] A_{k} \innp{\nabla f(x_{k+1}), y_{k} - x_{k+1}} + A_k f(x_{k+1}) + A_k D_f(y_k, x_{k+1}) + A_k (f(x_k) - f(y_k)),
    \end{aligned}
\end{align*}
where $\circled{6}$ holds by definition of $x_{k+1}$. Thus, $\circled{5}$ clearly holds since the right hand side of $\circled{6}$ equals the left hand side of $\circled{5}$, by definition of the Bregman divergence. Note that the second summand of the right hand side of $\circled{3}$ equals to minus the regret $R_t(x^\ast)$ of the online learning game with linear losses $a_i\innp{\nabla f(x_t), \cdot}$  and played points $v_i$  with respect to the comparator $x^\ast$. We defined $\widehat{R}_t(x^\ast)$ as any computable upper bound on  $R_t(x^\ast)$, and thus we obtain $\circled{4}$ above by this bound, the assumption on $y_{t-1}$, and $D_f(y_i, x_{i+1}) \geq 0$. 

Now, we define our primal-dual gap as 
\begin{equation}\label{eq:general_primal_dual_gap}
    G_t \defi f(x_t) - L_t = \frac{\widehat{R}_t(x^\ast)}{A_t},
\end{equation}
which, by construction, is an upper bound on the primal gap $f(x_t) -f(x^\ast)$.
Thus, an online learning algorithm whose regret, or a bound of it, is computable, provides us with a computable primal-dual gap, and convergence is linked to the value of the regret. We can now apply optimistic follow-the-regularized leader (OFTRL) or optimistic Mirror Descent (OMD) online learning algorithms and provide regret bounds, by using a not necessarily strongly convex or differentiable regularizer in order to obtain an optimistic generalized \FW{} algorithm, where we assume that we can solve subproblems of the form $\argmin_{x}\{\innp{w, x} + \psi(x) \}$, for any $w \in \R^d$. Typically in \FW{} algorithms, $\psi$ is just the indicator function of the feasible set. We provide an overview of these algorithms, that we have generalized to deal with subdifferentiable regularizers.

\subsection{Optimistc FTRL and optimistic MD with convex subdifferentiable regularizers}
Given a function $\psi$ that is subdifferentiable in a closed convex feasible set $\X$, two points $x, y \in \X$ and $\phi \in \partial \psi(y)$, define the non-differentiable Bregman divergence $D_{\psi}(x, y, \phi)$ as 
\[
    D_{\psi}(x, y; \phi) \defi \psi(x) - \psi(y) - \innp{\phi, x-y},
\] 

We start by presenting a regret bound for Optimistic FTRL, by using our possibly non-differentiable non-strongly convex regularizers.
\begin{theorem}[Optimistic FTRL]\label{thm:oftrl}
    Let $\X$ be a closed convex set and let $\psi_t, \ell_t: \X \to \R$ be closed, proper, convex and subdifferentiable functions in $\X$, for $t\geq 1$. For $t \in [T]$ and given some hints $\tilde{g}_t \in \Rd$, let ${z_t \defiin \argmin_{z \in \X} \{\sum_{i=1}^{t-1} \ell_i(z) + \innp{\tilde{g}_t, z} + \psi_t(z)\}}$, which we assume to exist. Also, define $z_{T+1} = u$ be an arbitrary point $u\in \X$. Then, the regret $R_T(u)$ satisfies:
    \[
        \sum_{t=1}^T ( \ell_t(z_t) - \ell_t(u) ) \leq \psi_{T+1}(u) - \min_{z\in \X}\psi_1(z) + \sum_{t=1}^T \Big(\innp{g_t - \tilde{g}_t, z_t - z_{t+1}} - D_{f_t}(z_{t+1}, z_t; g_t - \tilde{g}_t) + \psi_t(z_{t+1}) - \psi_{t+1}(z_{t+1}) \Big),
    \] 
    for all subgradients $g_t \in \partial \ell_t(z_t)$, where $f_t(z) \defi \sum_{i=1}^t \ell_i(z) + \psi_t(z)$.
\end{theorem}

\begin{proof}
    First, note that since $z_t = \argmin_{z \in \X}\{f_t(z) - \ell_t(z) + \innp{\tilde{g}_t, z}\}$, we have $0 \in \partial(f_t - \ell_t + \innp{\tilde{g}_t, \cdot})(z_t)$ and thus $g_t - \tilde{g}_t \in \partial f(z_t)$ for any $g_t \in \partial \ell_t(z_t)$, so the expression $D_{f_t}(z_{t+1}, z_t; g_t-\tilde{g}_t)$ above makes sense. We bound the regret as 
\begin{align}\label{eq:ftrl_analysis}
    \begin{aligned}
        \sum_{t=1}^{T} &(\ell_{t}(z_t) - \ell_{t}(u)) = \sum_{t=1}^{T} \left[ \left(\ell_{t}(z_t) + \sum_{i=1}^{t-1} \ell_{i}(z_t) + \psi_t(z_t)\right) - \left(\sum_{i=1}^t \ell_{i}(z_{t+1}) +  \psi_{t+1}(z_{t+1})\right) \right] \\
        & \ \ \ + \psi_{T+1}(u) - \psi_1(z_1)  \\
        &= \sum_{t=1}^{T} \left[ \left(\sum_{i=1}^{t} \ell_{i}(z_t) + \psi_t(z_t)\right) - \left(\sum_{i=1}^t \ell_{i}(z_{t+1}) + \psi_{t}(z_{t+1})\right)  + \psi_{t}(z_{t+1}) - \psi_{t+1}(z_{t+1}) \right]\\
        &\ \ \ + \psi_{T+1}(u) - \psi_1(z_1)\\
        &\leq \psi_{T+1}(u) - \min_{z\in \X}\psi_1(z) + \sum_{t=1}^T \Big(\innp{g_t - \tilde{g}_t, z_t - z_{t+1}} - D_{f_t}(z_{t+1}, z_t; g_t - \tilde{g}_t) + \psi_t(z_{t+1}) - \psi_{t+1}(z_{t+1}) \Big).
    \end{aligned}
\end{align}
    Above, we simply add and subtract terms and in the inequality we just bound $-\psi_1(z_1) \leq \min_{z\in\X}\psi_1(z)$ .
\end{proof}

\begin{corollary}\label{corol:oftrl}
    Under the assumptions of \cref{thm:oftrl}, for time-invariant $\psi_t = \psi$ we have
\begin{align}
    \begin{aligned}
        \sum_{t=1}^T ( \ell_t(z_t) - \ell_t(u) ) &\leq \psi_{T+1}(u) - \min_{z\in \X}\psi_1(z) + \sum_{t=1}^T \innp{g_t - \tilde{g}_t, z_t - z_{t+1}} - D_{f_t}(z_{t+1}, z_t; g_t - \tilde{g}_t) \\
        & \leq \psi_{T+1}(u) - \min_{z\in \X}\psi_1(z) + \sum_{t=1}^T \innp{g_t - \tilde{g}_t, z_t - z_{t+1}},
    \end{aligned}
\end{align}
\end{corollary}
\begin{proof}
    The result follows by noticing that the terms with $\psi_t$ in the sum cancel and that the Bregman divergences of the convex functions $f_t$ are non-negative.
\end{proof}

Now we present an alternative algorithm to the above, the optimistic Mirror Descent algorithm. First, we prove a lemma about the generic update rule of Mirror Descent:
$x_{t+1} \in \argmin_{x\in X} \left\{ \eta_t\innp{g, x} + D_{\psi}(x, x_t, \phi_t) \right\}$, where an assumption is made for $\psi$ that the $\argmin$ is always non empty. This holds for instance, for $\psi$ being the indicator function of a convex compact set, or $\psi$ being a Legendre function. The following corresponds to the classical mirror descent lemma, but for non-differentiable maps.

\begin{lemma}[Mirror Lemma with non-Differentiable Mirror Map]\label{lemma:mirror_lemma_non_diff}
    Given a closed convex set $\X$, let $\psi$ be proper, closed, convex and subdifferentiable in $\X$, and let $g\in \R^d$. If  $x_{t+1}\in\argmin_{x\in\X}\{\innp{g, x} + D_\psi(x, x_t; \phi_t) \}$ exists, then for all $u \in \X$ and all $\phi_{t+1} \in \partial \psi(x_{t+1})$:
    \[
        \innp{g, x_{t+1} - u} \leq D_{\psi}(u, x_t; \phi_t) - D_{\psi}(x_{t+1}, x_t; \phi_t)  - D_{\psi}(u, x_{t+1}; \phi_{t+1}).
    \] 
\end{lemma}
We note that by optimality of $x_{t+1}$, we have $0 \in g + \partial \psi(x_{t+1}) - \phi_t$ and so a possible choice of $\phi_{t+1}$ is $\phi_t - g$.

\begin{proof}
    The point $x_{t+1}$ is a minimizer of $F(x) \defi \innp{g, x} + D_\psi(x, x_t; \phi_t)$. If we substitute the definition of $F$ into the following expression, implied by the first order optimality condition of $x_{t+1}$, $F(x_{t+1}) + D_{F}(u, x_{t+1}; \phi_{t+1}) \leq F(u)$, then we obtain the lemma above. 
\end{proof}

The following theorem about mirror descent is a slight generalization over the common one using not necessarily differentiable regularizers and using a hint for the first step. Compare with, for instance, \citep[Theorem 6.20]{orabona2019modern} with constant step size.

\begin{theorem}[Optimistic MD]\label{thm:omd}
    Let $\X$ be a closed convex set and let $\psi, \ell_t: \X \to \R$ be closed, proper, convex and subdifferentiable functions in $\X$, for $t \in [T]$. For $t \in [T]$, and given some hints $\tilde{g}_t \in \Rd$ and $g_0 = \tilde{g}_0 = 0$, $z_0 \in \X$, let ${z_t \defiin \argmin_{z \in \X} \{ \innp{g_{t-1} + \tilde{g}_t - \tilde{g}_{t-1}, z} + D_\psi(z, z_{t-1}; \phi_{t-1})\}}$ for $\phi_{t-1} \in \partial \psi(z_{t-1})$, which we assume it exists. Also define $z_{T+1} = u$ be an arbitrary point $u\in \X$. Then, the regret $R_T(u)$ satisfies:
\begin{align}
    \begin{aligned}
        \sum_{t=1}^T ( \ell_t(z_t) - \ell_t(u) ) &\leq D_{\psi}(u, z_0; \phi_0) + \sum_{t=1}^T \Big(\innp{g_t - \tilde{g}_t, z_t - z_{t+1}} - D_{\psi}(z_{t+1}, z_t; \phi_t) \Big) - D_\psi(z_1, z_0; \phi_0) \\
        &\leq  D_{\psi}(u, z_0; \phi_0) + \sum_{t=1}^T \innp{g_t - \tilde{g}_t, z_t - z_{t+1}}.
    \end{aligned}
\end{align}
    for all subgradients $g_t \in \partial \ell_t(z_t)$, $t \geq 1$, where $z_{T+1}$ is defined as above with $\tilde{g}_{T+1} = 0$.
\end{theorem}

\begin{proof}
    Fix the choices $g_t \in \partial \ell_t(z_t)$ for $t\in [T]$ and define $\phi_t \defi g_{t-1} - \tilde{g}_t + \tilde{g}_{t+1}$. Applying \cref{lemma:mirror_lemma_non_diff} and adding a term with $z_{t}$, we obtain
    \begin{equation}\label{eq:aux:multiple_steps_omd}
        \innp{g_{t} +\tilde{g}_{t+1} - \tilde{g}_{t}, z_{t} - u} \leq D_\psi(u, z_{t}; \phi_{t}) - D_{\psi}(z_{t+1}, z_{t}, \phi_{t}) - D_\psi(u, z_{t+1}; \phi_{t+1}) + \innp{g_{t} +\tilde{g}_{t+1} - \tilde{g}_{t}, z_{t} - z_{t+1}}.
    \end{equation}
    Adding up the above from $t=0$ to $T$, and taking into account that $g_0 = \tilde{g}_0 = 0$, and setting $\tilde{g}_{T+1}$ we obtain an inequality whose left hand side is
\begin{align*}
    \begin{aligned}
        \sum_{t=0}^T &\innp{g_{t} +\tilde{g}_{t+1} - \tilde{g}_{t}, z_{t} - u} = \sum_{t=1}^T \innp{g_t, z_t - u} + \Ccancel[red]{\innp{\tilde{g}_{T+1} - \tilde{g}_0, u}} + \sum_{t=0}^T \innp{\tilde{g}_{t+1} -\tilde{g}_t, z_t}  \\
        &\circled{1}[\geq] \sum_{t=1}^T \Big(\ell_{t}(z_t) - \ell_t(u) \Big) + \sum_{t=0}^T \innp{\tilde{g}_{t+1}, z_t - z_{t+1}}.
    \end{aligned}
\end{align*}
    Note that we can set $\tilde{g}_{T+1}$ to any value since it does not play a role in any of the first $T$ predictions. For simplicity, we used $\tilde{g}_{T+1} = 0$. In $\circled{1}$ above, for the first summand we used the subgradient property, the second summand vanished by our choice of hints and then we used an equality for the last term, using again that $\tilde{g}_0 = \tilde{g}_{T+1}= 0$.

    Now if we combine the above with the left hand side of the result from adding up \cref{eq:aux:multiple_steps_omd} from $t=0$ to $T$, teslescoping and dropping $D_\psi(u, x_{T+1}; \phi_{T+1})$ we obtain
\begin{align*}
    \begin{aligned}
        \sum_{t=1}^T  \Big(\ell_{t}(z_t) - \ell_t(u) \Big) + \Ccancel[red]{\sum_{t=0}^T \innp{\tilde{g}_{t+1}, z_t - z_{t+1}}} &\leq D_\psi(u, z_0; \phi_0) - \sum_{t=0}^T D_\psi(z_{t+1}, z_t; \phi_t) + \sum_{t=0}^T \innp{g_{t} +\Ccancel[red]{\tilde{g}_{t+1}} - \tilde{g}_{t}, z_{t} - z_{t+1}} \\
        & \circled{1}[\leq] D_\psi(u, z_0; \phi_0) + \sum_{t=1}^T \innp{g_{t} - \tilde{g}_{t}, z_{t} - z_{t+1}},
    \end{aligned}
\end{align*}
    where in $\circled{1}$ we drop the Bregman divergence terms and we start the sum from $t=1$ since it is $0$ for $t=0$. The inequalities above equal the one in the statement. 

\end{proof}

\begin{remark}[Constant step size FTRL as an instance of OMD]\label{remark:equivalence_of_ftrl_and_omd}
It is well known that FTRL and Mirror Descent are equivalent for constant step sizes in the unconstrained setting, and so are their optimistic variants. With our non-differentiable extension we can see that actually in the constrained case these algorithms are also equivalent, for a specific choice of subgradients in the MD algorithm.
 Indeed, because of the optimality of $z_t$, we have $0 \in (g_{t-1} + \tilde{g}_t - \tilde{g}_{t-1}) + \partial\psi(z_t) - \phi_{t-1}$, and so $\phi_t \defi \phi_{t-1} - (g_{t-1} + \tilde{g}_t - \tilde{g}_{t-1}) \in \partial \psi(z_t)$ can be our next subgradient. In fact, if we choose $z_0 \in \argmin_{z\in\X} \psi(z)$, we can make the choice $\phi_0 = 0$, and using the recurrence in the definition of $\phi_t$, we obtain $\phi_t = -\tilde{g}_t - \sum_{i=1}^{t-1} g_i$ by using $g_0 = \tilde{g_0} = 0$. So the update rule becomes 
    \[
        z_t \in \argmin_{z\in\X} \{ \innp{g_{t-1} + \tilde{g}_t - \tilde{g}_{t-1}, z} + \psi(z) - \innp{\phi_{t-1}, z} \} = \argmin_{z\in\X} \left\{  \innp{ \tilde{g}_t + \sum_{i=1}^{t-1} g_i, z} + \psi(z)  \right\},
    \] 
    so under this choice the algorithm becomes equivalent to OFTRL with constant regularizer $\psi$. However, note that other choices of subgradients $\phi_{t}$ yield different update rules.
\end{remark}

\subsection{Convergence rate of optimistic FW}\label{sec:convergence_rate_optimistic_FW}

Given the theory developed so far in this section, we can now show the convergence of the two variants in our \cref{alg:optimistic_frank_wolfe}. We denote $G_t^{\operatorname{OP}}$ the primal-dual gap in \cref{eq:general_primal_dual_gap} when we use either of our two optimistic update rules and where the upper bound of the regret $R_t(x^\ast)$ is $\widehat{R}_t(x^\ast) \defi \max_{v\in \X} R_t(v)$.

\OptimisticFWGuarantees*

Generally \FW{} algorithms are applied to optimization problems with compact feasible sets, in which case $x^\ast$ above exists. However, the generalized \FW{} framework does not assume the feasible set is compact or that $f$ has a minimizer. We note that in the proof, the value of the parameters $a_t$ is not used until the last inequality in the theorem statement. One can also set $a_t = \Theta(t^c)$ for a constant $c > 1$, thus obtaining $A_t = \Theta(t^{c+1})$ and a rate $f(x_t) - f(u) = \bigo{ \frac{\psi(u)-\psi(x_1)}{t^{c+1}} + \frac{1}{t^{c+1}}\sum_{i=1}^t \frac{LD^2 i^{2c}}{i^{c+1}} } = \bigo{ \frac{\psi(u)-\psi(x_1)}{t^{c+1}} + \frac{LD^2}{t} }$. Thus, we can reduce the influence of $\psi$ on the convergence rate polynomially. 

We also note that the most common use of Frank-Wolfe corresponds to $\psi$ being the indicator function of a convex compact set $\X$, in which case we obtain $f(x_t) - f(x^\ast) \leq \frac{4LD^2}{t+1}$.

\begin{proof}\linkofproof{thm:optimistic_FW_guarantees}
    We start proving the statement for the choice of OFTRL in Line \ref{line:fw_oftrl} of \cref{alg:optimistic_frank_wolfe} that prescribes how the points $v_i$ are defined. Denoting $R_t(u)$ the corresponding regret at $u$ after $t$ steps, taking into account that we use $a_t g_i \defi a_i\nabla f(x_{i-1})$ as hints for $i\geq 1$, and defining $\tilde{v}_{i+1} = v_{i+1}$ for $i < t$ and $\tilde{v}_{t+1} = u$, we obtain
    \begin{equation}\label{eq:ineqs_of_opti_fw}
   \begin{split}
       f(x_t) - f(u) \leq G_t^{\operatorname{OP}}  &\circled{1}[\leq]\frac{R_t(u)}{A_t} \\
           &\circled{2}[\leq] \frac{\psi(u)-\psi(x_1)}{A_t} + \frac{1}{A_t}\sum_{i=1}^t a_i \innp{\nabla f(x_i) - g_i, v_i -\tilde{v}_{i+1}} \\
           &\circled{3}[\leq] \frac{\psi(u)-\psi(x_1)}{A_t}  + \frac{1}{A_t}\sum_{i=1}^t a_i L\norm{x_i-x_{i-1}}D \\
           &\circled{4}[\leq] \frac{\psi(u)-\psi(x_1)}{A_t}  + \frac{1}{A_t}\sum_{i=1}^t \frac{a_i^2LD^2}{A_i}\\
           &\circled{5}[\leq]\frac{\psi(u)-\psi(x_1)}{t(t+1)} + \frac{1}{t(t+1)}\sum_{i=1}^t\frac{4iLD^2}{i+1}  \\
           &\circled{6}[<]\frac{\psi(u)-\psi(x_1)}{t(t+1)} + \frac{4 LD^2}{t+1}.
  \end{split}
\end{equation}
    where $\circled{1}$ uses \cref{eq:general_primal_dual_gap}, which holds since our algorithm has the form described at the beginning of \cref{sec:optimistic_proofs} for the points $v_t$ running the OFTRL algorithm. Inequality $\circled{2}$ holds by \cref{corol:oftrl}, since our points $v_i$ are computed according to the OFTRL algorithm. Then, $\circled{3}$ uses $g_t = \nabla f(x_{t-1})$, $L$-Lipschitzness of $\nabla f(\cdot)$, the general Cauchy-Schwarz inequality and bounds $\norm{v_i - \tilde{v}_{i+1}} \leq D$ for all $i\in[t]$, and $\circled{4}$ uses that by definition of $x_i$, we have $x_{i} - x_{i-1} = \frac{a_i}{A_i} (v_i - x_{i-1})$ for $i\geq 1$ and bounds $\norm{v_i - x_{i-1}}\leq D$. This yields the first part of the theorem statement. Now, substituting the choices of $a_t = 2t$, and thus $A_t = \sum_{i=0}^t a_i =  t(t+1)$ we obtain $\circled{5}$, and a simple bound gives $\circled{6}$.

    The proof for the OMD variant in Line \ref{line:fw_omd} of \cref{alg:optimistic_frank_wolfe} is identical, except that in $\circled{2}$ above we use \cref{thm:omd} and so $\tilde{v}_{t+1}$ changes to be a point in $\argmin_{v\in\Rd}\{\innp{g_{t} - \tilde{g}_{t}, v} + D_\psi(z, v_{t}; \phi_{t})\}\}$, that is, it corresponds to the next point computed with the update rule when we choose $\tilde{g}_{t+1} = 0$. And also, we have $D_{\psi}(u, x_0)$ instead of $\psi(u) - \psi(x_1)$. The rest of the inequalities in \cref{eq:ineqs_of_opti_fw} are the same.

    Above, we bounded $\norm{v_i - x_{i-1}} \cdot \norm{v_i - \tilde{v}_{i+1}} \leq D^2$. If we assume $f$ satisfies the cocoercive inequality $D_f(x, y) \geq \frac{1}{L}\norm{\nabla f(x) - \nabla f(y)}^2$, which holds if $f$ is convex and smooth in $\Rd$ and not just in $\dom(\psi)$, cf. \citep[Theorem 2.1.5]{nesterov2018lectures}, we can obtain an alternative proof where we will end up having a similar bound, but the $D^2$ comes from $\norm{v_i - \tilde{v}_{i+1}}^2$, except for a term in the first iteration. We do it for OFTRL only for simplicity. Indeed, using the right hand side of $\circled{3}$ in \cref{eq:lower_bound_generic_anytime_online_to_batch} as lower bound, we have some extra terms that allow for this task
\begin{equation}\label{eq:ineqs_of_opti_fw_md}
   \begin{split}
       f(x_t) - f(u) \leq G_t &\circled{1}[\leq]\frac{R_t(u)}{A_t} \\
           &\circled{2}[\leq] \frac{\psi(u)-\psi(x_1)}{A_t} + \frac{1}{A_t}\left( \sum_{i=1}^t a_i \innp{\nabla f(x_i) - g_i, v_i -\tilde{v}_{i+1}} - A_{i-1} D_f(x_{i-1}, x_{i})  \right)  \\
           &\circled{3}[\leq] \frac{\psi(u)-\psi(x_1)}{A_t}  + \frac{1}{A_t} \left( a_1\norm{\nabla f(x_1)- \nabla f(x_0)}D + \sum_{i=2}^t \frac{a_i^2 L}{A_{i-1}}\norm{v_i-\tilde{v}_{i+1}}^2  \right)\\
           &\circled{4}[\leq] \frac{\psi(u)-\psi(x_1)}{A_t}  + \frac{1}{A_t} \left( 2LD^2 + \sum_{i=2}^t \frac{4i LD^2}{i-1}  \right)\\
           &\circled{5}[\leq]\frac{\psi(u)-\psi(x_1)}{t(t+1)} + \frac{LD^2(2+ 4(t-1) + 4\log(t))}{t(t+1)}\\
           &\circled{6}[=]\bigol{\frac{\psi(u)-\psi(x_1)}{t(t+1)} + \frac{LD^2}{t}}.
  \end{split}
\end{equation}
\end{proof}

\begin{remark}[Computable primal-dual gap]\label{remark:computable_optimistic_primal_dual_gap}
    Note that for $\psi$ being the indicator of a compact set, and for $u = x^\ast \in \argmin_{x\in \operatorname{dom}(\psi)} f(x)$, we have for the OFTRL variant the upper bound on the primal-dual gap $\frac{1}{A_t}\sum_{i=1}^t a_i \innp{\nabla f(x_i) - g_i, v_i -\tilde{v}_{i+1}} - A_{i-1}D_f(x_{i-1}, x_i)$, which depends on the unknown point $x^\ast$, since $\tilde{v}_{t+1} = x^\ast$. For a computable primal-dual gap, we can apply an analogous bound to $\circled{3}$ in \cref{eq:ineqs_of_opti_fw} but for the last summand only, that is
    \[
        G_t \leq \frac{1}{A_t} \left( \sum_{i=1}^{t-1} a_i \innp{\nabla f(x_i) - g_i, v_i -v_{i+1}} + \norm{\nabla f(x_t) - \nabla f(x_{t-1})}_\ast D - \sum_{i=1}^t A_{i-1}D_f(x_{i-1}, x_i) \right).
    \] 
    Alternatively, using OMD and taking one more linear minimization oracle to compute $\tilde{v}_{t+1}$, we already have that our bound is a computable primal-dual gap: $\frac{1}{A_t}\sum_{i=1}^t a_i \innp{\nabla f(x_i) - g_i, v_i -\tilde{v}_{i+1}} - A_{i-1}D_f(x_{i-1}, x_i)$. We note that it is also possible to obtain an analogous different $\tilde{v}_{t+1}$ for FTRL that does not depend on $x^\ast$, at the expense of computing a linear minimization oracle. Just take the equivalence of OFTRL and OMD for specific choices of the subgradients $\phi_t$ in \cref{remark:equivalence_of_ftrl_and_omd}.
\end{remark}

\section{Proofs for Primal-Dual Short-Steps}

\begin{proof}\linkofproof{prop:primal_dual_steps}
    In \cref{eq:primal_dual_good} or \cref{eq:primal_dual_good_hb}, after their respective inequalities $\circled{1}$, isolating $G_t$, using $\gamma_t \defi \frac{a_t}{A_t}$, which ranges from $[0, 1)$ as $a_t \in [0, \infty)$, we obtain 
\begin{align*}
     \begin{aligned}
         G_t &\leq \frac{A_{t-1}G_{t-1}}{A_t} - \frac{a_t}{A_t} \innp{\nabla f(x_t), v_t-x_t}  + (f(x_{t+1}) - f(x_t))  \\
         &= (1-\gamma_t)G_{t-1} - \gamma_t \innp{\nabla f(x_t), v_t-x_t}  + f((1-\gamma_t)x_t + \gamma_t v_t) - f(x_t).
     \end{aligned}
    \end{align*}
Differentiating the right hand side twice with respect to $\gamma_t$, we obtain
\[
    \innp{\nabla^2 f((1-\gamma_t)x_t + \gamma_t v_t) (v_t-x_t), v_t-x_t} \geq 0,
\] 
hence the expression is convex with respect to $\gamma_t$, which proves the first statement. 

    For the second one, we already performed some steps of the proof in the main paper. Using the primal-dual analysis of \FW{} and \HB{} \cref{sec:review_of_primal_dual_fw} and a few computations we arrive to \cref{eq:isolating_primal_dual_gap}.  Plugging the choice of our primal-dual short step \cref{eq:primal-dual-short} into \cref{eq:isolating_primal_dual_gap} and simplifying leads to 
\begin{align*}
	G_t \leq \left(1 - \frac{\gamma_t}{2}\right) G_{t-1} = \left(1 - \frac{\min\{1,\frac{G_{t-1}}{L \norm{v_{t}-x_t}^2}\}}{2}\right) G_{t-1},
\end{align*}
or equivalently
\begin{align*}
	G_{t-1} - G_t \geq \frac{1}{2} G_{t-1} \min\{1,\frac{G_{t-1}}{L \norm{v_{t}-x_t}^2}\}.
\end{align*}
    Now, one can apply \cite[Lemma 2.21]{CGFWSurvey2022} (or similar results; see \citep{garber2015faster}), which converts the contraction inequality into a convergence guarantee, so we obtain:
$$
G_t \leq \frac{4LD^2}{t+2},
$$
    for the primal-dual gap convergence rate after bounding $\norm{v_i - x_i}^2 \leq D^2$ for all $i \leq t$.

    The progress made by the line search in terms of primal-dual gap is greater than the guaranteed progress \cref{eq:isolating_primal_dual_gap} used by this second approach and so the line search variant also enjoys the same rates of convergence. 
\end{proof}

\begin{remark}[Using $f(x^\ast)$ for the gap]\label{remark:pd_short_steps_are_a_generalization}
    The primal-dual short step is a generalization of the standard short steps, since if we choose the best possible lower bound $L_t \defi f(x^\ast)$, which is a value that we do not know in general, and if we define the gap accordingly $G_t \defi f(x_{t+1}) - f(x^\ast)$, then we obtain
    \[
    A_t G_t - A_{t-1}G_{t-1} = A_t f(x_{t+1}) - A_{t-1} f(x_t) - a_t f(x^\ast),
    \] 
    which after using smoothness and reorganizing yields
\begin{align}
     \begin{aligned}
         G_t  &\leq (1-\gamma_t)G_{t-1} +  \gamma_t(f(x_t)-f(x^\ast))  + \gamma_t \innp{\nabla f(x_t), v_t-x_t} + \gamma_t^2\frac{L}{2} \norm{v_t-x_t}^2\\
         &\circled{1}[=] G_{t-1} + \gamma_t \innp{\nabla f(x_t), v_t-x_t} + \gamma_t^2\frac{L}{2} \norm{v_t-x_t}^2
     \end{aligned}
    \end{align}
    where $\circled{1}$ holds by definition of the gap $G_{t-1} = f(x_t)-f(x^\ast)$.
    Note that optimizing the right hand side of the last expression results into regular short steps, which can be computed even if we do not know the value of $L_t \defi f(x^\ast)$. This computation is simply saying the natural fact that if our primal-dual gap becomes the primal gap, then these new short steps that greedily maximize guaranteed progress on $G_t$, become the regular short steps, that greedily maximize primal progress.
\end{remark}

\subsection{Details on the primal-dual step size for Gradient Descent}\label{sec:primal_dual_steps_gd}

    First recall a few properties of the classical gradient descent (\GD{}) algorithm, with arbitrary step sizes $a_t$:
\begin{align*}
    \begin{aligned}
        x_{t+1} \defi x_{t}-a_t \nabla f(x_t) &=  \argmin_{x\in\R^d} \left\{ a_t\innp{\nabla f(x_t), x-x_t} + \frac{1}{2}\norm{x-x_t}_2^2 \right\} \\ 
        &= \argmin_{x\in\R^d} \left\{ \sum_{i=1}^t a_i\innp{\nabla f(x_i), x-x_i} + \frac{1}{2}\norm{x-x_1}_2^2 \right\}  = x_0 - \sum_{i=1}^t a_i \nabla f(x_i).
    \end{aligned}
\end{align*}

\begin{proof}\linkofproof{prop:gd_primal_dual_steps}

    Recall our notation of positive weights $a_t$ and $A_t \defi \sum_{i=1}^t a_i$. The lower bound that we use to define the primal-dual gap is obtained from
\begin{align*}
     \begin{aligned}
         A_t f(x^\ast) &\circled{1}[\geq] \sum_{i=1}^t a_i f(x_i) + \sum_{i=1}^t a_i \innp{\nabla f(x_i), x^\ast - x_i} \pm \frac{1}{2}\norm{x^\ast- x_1}_2^2  \\
         &\circled{2}[\geq]  \sum_{i=1}^t a_i f(x_i) + \min_{x\in\R^d}\left\{\sum_{i=1}^t a_i \innp{\nabla f(x_i), x - x_i} + \frac{1}{2}\norm{x-x_1}_2^2\right\} - \frac{1}{2}D^2 \\ 
         &\circled{3}[=]  \sum_{i=1}^t a_i f(x_i) + \sum_{i=1}^t a_i \innp{\nabla f(x_i), x_{t+1} - x_i} + \frac{1}{2}\norm{x_{t+1}-x_1}_2^2 - \frac{1}{2}D^2 \defi A_t L_t. 
     \end{aligned}
    \end{align*}
    Where $\circled{1}$ is due to convexity, and in $\circled{2}$ and $\circled{3}$ we take a minimum and use the definition of $x_{t+1}$ as a minimizer of $\Lambda_t(x)\defi \sum_{i=1}^t a_i \innp{\nabla f(x_i), x-x_i}+\frac{1}{2}\norm{x-x_1}_2^2$. We also used the bound $D \geq \norm{x^\ast - x_1}_2$. Now define the gap $G_t \defi f(x_{t+1}) - L_t$. We have
\begin{align}\label{eq:gd_adgt_analysis}
     \begin{aligned}
         A_t& G_t - A_{t-1}G_{t-1} - \eventindicator{t=1}\frac{1}{2}D^2 \circled{1}[=] A_t (f(x_{t+1})- f(x_{t})) + \Ccancel[red]{a_t f(x_t)}  \\
         &\quad - \left( \Ccancel[red]{\sum_{i=1}^{t} a_i f(x_i)} + \sum_{i=1}^{t-1} a_i  \innp{\nabla f(x_i), x_{t+1}-x_i} + \frac{1}{2}\norm{x_{t+1}- x_1}_2^2 \right)  - a_t\innp{\nabla f(x_t), x_{t+1}-x_t}  \\
         &\quad + \left( \Ccancel[red]{\sum_{i=1}^{t-1} a_i f(x_i)} + \sum_{i=1}^{t-1} a_i \innp{\nabla f(x_i), x_{t}-x_i} + \frac{1}{2}\norm{x_{t}- x_1}_2^2  \right)\\
         &\circled{2}[\leq] \innp{\nabla f(x_t), A_t (x_{t+1}-x_t) - a_t(x_{t+1} - x_t)} + \frac{A_tL}{2}\norm{x_{t+1}-x_t}_2^2  - \frac{1}{2}\norm{x_{t+1}-x_t}_2^2 \\
         &\circled{3}[=] \norm{\nabla f(x_t)}_2^2 \left(- a_t A_{t-1} + \frac{a_t^2A_tL}{2} - \frac{a_t^2}{2}\right) \defi E_t(a_t).
     \end{aligned}
    \end{align}
    Above, in $\circled{1}$ we write out the definitions and cancel some terms. In $\circled{2}$, we used smoothness for the first term, and we used the $1$-strong convexity of $\Lambda_{t-1}(x)$ and the fact that $x_t$ is its minimizer so $\Lambda_{t-1}(x_t) - \Lambda_{t-1}(x_{t+1}) \leq -\frac{1}{2}\norm{x_{t+1}-x_t}_2^2$. In $\circled{3}$, we used the definition of $x_{t+1}$ and grouped terms. This time, we have defined the error $E_t(a_t)$ as a function of $a_t$.

    We can apply the same technique as in the primal dual steps for Frank-Wolfe algorithms. Let $\mathcal{G}_{t-1} \defi A_{t-1}G_{t-1} + \eventindicator{t=1} \frac{1}{2}D^2 \geq A_{t-1}(f(x_{t+1}) - f(x^\ast)) \geq 0$. Hence, for $t\geq 1$, we aim to  minimize the right hand side of the following  that comes reorganizing the above
    \begin{equation}\label{eq:short_primal_dual_gd_ineq}
        G_t \leq \frac{\mathcal{G}_{t-1}}{A_t} +\norm{\nabla f(x_t)}_2^2 \left(- \frac{a_t A_{t-1}}{A_t} + \frac{a_t^2L}{2} - \frac{a_t^2}{2A_t}\right).
    \end{equation}
    Differentiating the RHS with respect to $a_t$ and equating to $0$ (recall $A_t = A_{t-1}+a_t$) we obtain
\begin{align}\label{eq:optimizing_gd_primal_dual_short_step}
     \begin{aligned}
         0 &= - \frac{\mathcal{G}_{t-1}}{A_t^2}
         + \norm{\nabla f(x_t)}_2^2\left(-\frac{A_{t-1}A_t - a_t A_{t-1}}{A_t^2} + a_t L - \frac{ 2a_tA_t - a_t^2}{2A_t^2}  \right)\\
         \iff 0 &= - 2\mathcal{G}_{t-1} + \norm{\nabla f(x_t)}_2^2\left( -2A_{t-1}^2 + 2 a_t A_t^2 L -  2a_tA_t + a_t^2  \right),
     \end{aligned}
    \end{align}
    which gives a cubic equation for $a_t$. From now on, assume $a_t$ has the value of this minimizer.
Note that this solution is always greater than $\frac{1}{2L}$, for all $t \geq 1$. Indeed, first notice that the right hand side of \eqref{eq:short_primal_dual_gd_ineq} is convex on $a_t$ for $a_t > 0$, since if we differentiate the right hand side a second time we obtain:
    \[
        2 \frac{\mathcal{G}_{t-1}}{A_t^3}  + \norm{\nabla f(x_t)}_2^2\left(\frac{2A_{t-1}^2}{A_t^3} + L - \frac{A_{t-1}^2}{A_t^3} \right) \geq 0.
    \] 
    Given this convexity, in order to show that the solution of \cref{eq:optimizing_gd_primal_dual_short_step} is $a_t > \frac{1}{2L}$, it is enough to check that the derivative is negative at $a_t = \frac{1}{2L}$ which is immediate after substitution. This fact yields the convergence rate:
\begin{align*}
     \begin{aligned}
         f(x_{t+1}) - f(x_t)  \leq    G_t &\circled{1}[\leq] \frac{\mathcal{G}_{t-1}}{A_t} +\norm{\nabla f(x_t)}_2^2 \left(- \frac{a_t A_{t-1}}{A_t} + \frac{a_t^2L}{2} - \frac{a_t^2}{2A_t}\right) \\
             &\circled{2}[\leq] \frac{\mathcal{G}_{t-1}}{A_{t-1} + 1/(2L)} \circled{3}[\leq] \frac{\mathcal{G}_{0}}{t/(2L)}\\
             &\circled{4}[=] \frac{LD^2}{t},
     \end{aligned}
    \end{align*}
    where $\circled{1}$ is \cref{eq:short_primal_dual_gd_ineq}, $\circled{2}$ holds by substituting the minimizer $a_t$ by $1 /(2L)$ and droping $E_t(1 / (2L))$ which after substitution it is immediate to see that it is nonpositive. Inequality $\circled{3}$ holds by applying recursively $\circled{1}$ and $\circled{2}$, taking into account that $\mathcal{G}_{k} = A_k G_k$ for all $k > 1$. And finally $\circled{4}$ substitutes the value of $\mathcal{G}_{0}$.

\paragraph{Line search for minimizing the primal-dual gap}

    Recall that we define $\mathcal{G}_{t-1} \defi A_{t-1}G_{t-1} + \eventindicator{t=1} \frac{1}{2}D^2 \geq 0$.
If for $t\geq 1$, after $\circled{1}$ in \cref{eq:gd_adgt_analysis} we do not apply smoothness but only apply the inequality $\Lambda_{t-1}(x_t) - \Lambda_{t-1}(x_{t+1}) \leq -\frac{1}{2}\norm{x_{t+1}-x_t}^2$, use $x_{t+1}-x_t = -a_t \nabla f(x_t)$ and then isolate $G_t$, we obtain
\begin{equation}\label{eq:without_using_smoothness}
    G_t \leq \frac{\mathcal{G}_{t-1}}{A_t} + \frac{a_t^2}{2A_t}\norm{\nabla f(x_t)}^2 + (f(x_{t+1})-f(x_t)) 
\end{equation}
If we differentiate twice with respect to $a_t$, we obtain (taking into account $A_t$ and $x_{t+1}$ depend on $a_t$):
\begin{equation*}
  2 \frac{\mathcal{G}_{t-1}}{A_{t}^3} +
  \frac{A_{t-1}^2}{A_t^3}\norm{\nabla f(x_t)}^2 + \innp{\nabla^2f(x_{t+1}) \nabla f(x_t), \nabla f(x_t)} \geq 0.
\end{equation*}
Thus, the right hand side of \cref{eq:without_using_smoothness} is convex with respect to the variable $a_t$, which means that we can do a line search in order to greedily minimize the primal-dual gap progress at each iteration. Above, we used twice differentiability of $f$ for the last summand but in fact this is not required, since $f(x_{t+1}) = f(x_{t}-a_t \nabla f(x_t))$ is clearly convex on $a_t$ due to the convexity of $f$ restricted to the line $a_t \mapsto x_{t}-a_t \nabla f(x_t)$.

Note that the upper bound in \cref{eq:without_using_smoothness} is tighter than the one in \cref{eq:short_primal_dual_gd_ineq}, and thus, the convergence rate from the line search on the primal-dual gap bound is at least the one from the primal-dual short step that minimizes \cref{eq:short_primal_dual_gd_ineq}.
\end{proof}

\section{Extending Primal Analyses of Known \FW{} Convergence Results to Primal-Dual Analyses}\label{sec:pd_extensions_fw}

In this section we demonstrate via some examples how primal-dual analyses can be derived from existing primal analyses. We picked examples that demonstrate relevant aspects of the argument, so that the interested reader should be able to carry them over to other settings. We also refer the interested reader to \citet{CGFWSurvey2022} for full details on the primal convergence analyses.

In line with \cref{sec:preliminaries}, we let $g(x_\ell) = \max_{v\in\X} \innp{\nabla f(x_\ell), x_\ell-v} = \innp{\nabla f(x_\ell), x_\ell-v_\ell}$ be the \FW{} gap and 
$$
A_\ell L_\ell = \sum_{i=0}^\ell a_i f(x_i) + \sum_{i=0}^\ell a_i \innp{\nabla f(x_i), v_i - x_i},
$$
be the associated lower bound function for the standard \FW{} gap. Together with $G_\ell = f(x_{\ell+1}) - L_\ell$ we can now rewrite the crucial term $A_\ell G_\ell - A_{\ell-1} G_{\ell - 1} \leq E_\ell$ appearing in primal-dual analyses as follows:
\begin{equation*}
 \begin{split}
  \MoveEqLeft
  A_\ell (f(x_{\ell + 1}) - L_\ell) - A_{\ell-1}(f(x_{\ell}) - L_{\ell-1})
  \\
  &
  = A_\ell f(x_{\ell + 1}) - A_{\ell-1} f(x_{\ell}) - (A_\ell L_\ell -
  A_{\ell-1} L_{\ell-1})
  \\
  &
  = A_\ell (f(x_{\ell + 1}) - f(x_{\ell})) + a_\ell g(x_\ell) \leq E_\ell.
 \end{split}
\end{equation*}
This can be rearranged to the following fundamental bound inequality:
\begin{equation}
	\label{fw:protoBaseIneq}
	\underbrace{a_\ell g(x_\ell)}_{\text{weighted gap}} - \underbrace{A_\ell  (f(x_\ell) - f(x_{\ell+1}))}_{\text{weighted progress}} \leq E_\ell.
\end{equation}

From this inequality there are various ways to procced. The most natural way is often to plug-in the progress guarantee from the short-step for a direction $d_t$ (typically $d_t = x_t - v_t$ but not always), that is, we use smoothness and optimize the bound:
$$
f(x_{\ell + 1}) - f(x_{\ell}) \leq - \frac{\innp{\nabla f(x_\ell), d_\ell}^2}{2L\norm{d_\ell}^2}.
$$
Chaining this bound into \cref{fw:protoBaseIneq} makes the left-hand side only larger, so that we obtain:
\begin{equation}
	\label{fw:baseIneq}
	\underbrace{a_\ell g(x_\ell)}_{\text{weighted gap}} - \underbrace{A_\ell \frac{\innp{\nabla f(x_\ell), d_\ell}^2}{2L\norm{d_\ell}^2}}_{\text{weighted progress}} \leq E_\ell.
\end{equation}

Inequalities \cref{fw:protoBaseIneq} and \cref{fw:baseIneq} will be key in the following to simply transfer primal convergence results into primal-dual convergence results. To this end the following lemma will be useful, which turns the relation between $a_\ell$ and $A_\ell$ into a convergence rate.

\begin{lemma}[Linear rate conversion]
	\label{lem:linRateCon}
	Suppose that $a_\ell (\kappa - 1) = A_{\ell -1 }$ (or equivalently: $a_\ell \kappa = A_{\ell}$) then it holds
	$$
	1/A_\ell = \left(1 - \frac{1}{\kappa}\right) 1/A_{\ell-1}.
	$$
	\begin{proof}
		The proof is by straightforward rearranging:
		\begin{align*}
			& a_\ell (\kappa - 1) = A_{\ell -1 } \\
			\Rightarrow\ & A_\ell (\kappa - 1) - A_{\ell -1} (\kappa - 1) = A_{\ell -1 } \\
			 \Rightarrow\ & A_\ell (\kappa - 1) = A_{\ell -1} \kappa \\
			 \Rightarrow\ & A_\ell = \frac{\kappa}{\kappa - 1} A_{\ell -1} \\
			 \Rightarrow\ & 1/A_\ell = \left(1 - \frac{1}{\kappa}\right) 1/A_{\ell-1}
		\end{align*}
	\end{proof}
\end{lemma}

\subsection{Linear convergence rates: the case $E_\ell = 0$}

We will first consider the case where we can set $E_\ell = 0$. This case, which usually comes with a cleaner analysis basically captures all linear convergence rate results for \FW{} algorithms. To this end we present a generic argument that then allows us to simply reuse already known analyses for primal convergence to establish primal-dual convergence simply by plugging in the bounds on the primal progress and the dual gap; it is important to note that these are usually known already from the primal analysis. We start from \eqref{fw:protoBaseIneq}: 
$$
a_\ell g(x_\ell) - A_\ell  (f(x_\ell) - f(x_{\ell+1})) \leq E_\ell, 
$$
and we set $E_\ell = 0$ and solve for equality. For the sake of exposition here and in the following let $p_\ell \doteq f(x_\ell) - f(x_{\ell+1})$ denote the \emph{primal progress} in iteration $\ell$. We obtain:
\begin{align}
	a_\ell g(x_\ell)  - A_\ell p_\ell = 0 \\
	\Rightarrow a_\ell g(x_\ell)  & = A_\ell p_\ell \\
	\Rightarrow a_\ell \left(g(x_\ell) - p_\ell \right) & = A_{\ell-1} p_\ell \\	 
\label{fw:generic_al}	\Rightarrow a_\ell \left(\frac{g(x_\ell)}{p_\ell} - 1\right) & = A_{\ell-1},
\end{align}
and combining this with \cref{lem:linRateCon}, we obtain:

\begin{lemma}[Transfer template for bounds]
	\label{lem:transferBounds}
	Let $g(x)$ be a dual bound and let $p_\ell$ denote the primal progress in iteration $\ell$. If we set $E_\ell = 0$, then we obtain:
	$$
	1/A_\ell = \left(1 - \frac{p_\ell}{g(x_\ell)}\right) 1/A_{\ell-1}.
	$$
\end{lemma}

Also observe that we can use \eqref{fw:generic_al} to dynamically compute the $a_\ell$ from an actual step which allows us to report an adaptive gap if desired.

\subsubsection{$f$ strongly convex and $x^* \in \operatorname{rel.int}(P)$}
\label{sec:interiorCase}

For the sake of completeness, we will present two ways of deriving primal-dual convergence rates for the case where $f$ is strongly convex and $x^* \in \operatorname{rel.int}(P)$. First, we will go the ``complicated'' route, without relying on \cref{lem:transferBounds} and instead directly solving \cref{fw:baseIneq} for $a_\ell$. Then, we will use \cref{lem:transferBounds} to obtain the (same) convergence rate. The primal analysis is due to \citet{guelat1986some}.

We assume that $f$ is $\mu$-strongly convex and $B(x^\ast, r) \subseteq \X$. We start from the primal gap bound via strong convexity:
$$
f(x) - f(x^*) \leq \frac{\innp{\nabla f(x), x - x^*}^2}{2 \mu \norm{x - x^*}^2},
$$
where $\mu > 0$ is the strong convexity constant, which we then combine with the \emph{scaling inequality} for $x^* \in \operatorname{rel.int}(P)$ from \citet{guelat1986some}:
$$
\frac{r}{D} \frac{\innp{\nabla f(x), x - x^*}}{\norm{x - x^*}} \leq \frac{\innp{\nabla f(x), x - v}}{\norm{x - v}},
$$
where $v$ is the FW-vertex for $x$, $D$ is the diameter, and $r$ is radius of the ball around the optimal solution $x^*$ contained in $P$, to obtain the bound on the primal gap:
$$
f(x) - f(x^*) \leq \frac{D^2}{r^2}\frac{\innp{\nabla f(x), x - v}^2}{2 \mu \norm{x - v}^2},
$$
and we set the gap function $g(x_\ell)$ to be the right hand side of the above inequality:
$$
g(x_\ell) \doteq \frac{D^2}{r^2}\frac{\innp{\nabla f(x_\ell), x_\ell - v_\ell}^2}{2 \mu \norm{x_\ell - v_\ell}^2},
$$
where $x_\ell, v_\ell$ are the iterate and \FW{}-vertex in the $\ell$-th iteration, respectively. Now for the direction $d_\ell$ we pick the standard \FW{} direction, i.e., $d_\ell = x_\ell - v_\ell$. With this choice \eqref{fw:baseIneq} becomes

\begin{align*}
	a_\ell g(x_\ell) - A_\ell \frac{\innp{\nabla f(x_\ell), d_\ell}^2}{2L\norm{d_\ell}^2} & \leq E_\ell \\
	\Rightarrow a_\ell \frac{D^2}{r^2}\frac{\innp{\nabla f(x_\ell), x_\ell - v_\ell}^2}{2 \mu \norm{x_\ell - v_\ell}^2} - A_\ell \frac{\innp{\nabla f(x_\ell), x_\ell - v_\ell}^2}{2L\norm{x_\ell - v_\ell}^2} & \leq E_\ell \\
	\Rightarrow \frac{\innp{\nabla f(x_\ell), x_\ell - v_\ell}^2}{\norm{x_\ell - v_\ell}^2} \left( a_\ell \frac{1}{2\mu} \frac{D^2}{r^2}  - A_\ell \frac{1}{2L} \right) & \leq E_\ell.
\end{align*}

We want $a_\ell \frac{1}{2\mu} \frac{D^2}{r^2}  - A_\ell \frac{1}{2L} = 0$, and in particular then we can choose $E_\ell = 0$. We obtain:

\begin{align*}
	a_\ell \frac{1}{2\mu} \frac{D^2}{r^2}  - A_\ell \frac{1}{2L} & = 0 \\
	\Rightarrow a_\ell \frac{1}{2\mu} \frac{D^2}{r^2} & = A_\ell \frac{1}{2L} \\
	\Rightarrow a_\ell \left(\frac{1}{2\mu} \frac{D^2}{r^2} - \frac{1}{2L} \right) & = A_{\ell-1} \frac{1}{2L} \\	 
	\Rightarrow a_\ell \left(\frac{L}{\mu} \frac{D^2}{r^2} - 1 \right) & = A_{\ell-1},
\end{align*}
and applying \cref{lem:linRateCon}, we obtain immediately that
$$
1/A_\ell = \left(1 - \frac{\mu}{L} \frac{r^2}{D^2}\right) 1/A_{\ell-1},
$$
and hence the expected linear rate of convergence follows for the primal-dual gap $G_t$, i.e.,
$$
G_t \leq \frac{A_{t-1}}{A_t} G_{t-1} = \left(1 - \frac{\mu}{L} \frac{r^2}{D^2}\right) G_{t-1}.
$$

Alternatively, we could have directly applied \cref{lem:transferBounds} to
$$
p_\ell = \frac{\innp{\nabla f(x_\ell), x_\ell - v_\ell}^2}{2L\norm{x_\ell - v_\ell}^2} \qquad \text{and} \qquad g(x_\ell) = \frac{D^2}{r^2}\frac{\innp{\nabla f(x_\ell), x_\ell - v_\ell}^2}{2 \mu \norm{x_\ell - v_\ell}^2},
$$
to obtain the same result.

\subsubsection{Away-step variants over polytopes and strongly convex functions}
Let $P$ be a polytope and $f$ be $\mu$-strongly convex. The proof is similar to \cref{sec:interiorCase}. We pick 
$$
g(x_\ell) \doteq \frac{\innp{\nabla f(x_\ell), s_\ell - v_\ell}^2}{2\mu w(P)^2},
$$
which is the \emph{geometric strong convexity} that arises from combining the strong convexity bound in \cref{sec:interiorCase} with a different scaling inequality, namely the one that appears in active set based, away-step inducing variants:
$$
\innp{\nabla f(x_\ell), s_\ell - v_\ell} \geq w(P) \frac{\innp{\nabla f(x_\ell), x_\ell - x^*}}{\norm{x_\ell - x^*}},
$$
where $w(P)$ is the \emph{pyramidal width} of $P$. Moreover, we pick $d_\ell = s_\ell - v_\ell$ and use a slightly different bound on the primal progress, namely
$$
p_\ell = \frac{\innp{\nabla f(x_\ell), s_\ell - v_\ell}^2}{4L\norm{s_\ell - v_\ell}^2},
$$
note the extra factor of $2$ in the denominator, which is due to the selection between the away-step or the \FW{} step; we refer the reader to \citet{lacoste2015global} and \citet{CGFWSurvey2022} for more details on both the geometric strong convexity as well as the modified primal progress bound; see also \citet{pena2018polytope} for a unified perspective on the different notions of geometric conditioning arising from the polytope. We also ignore drop steps here for the sake of simplicity, however they can be added easily and deteriorite the number of required steps only by a constant factor of $2$; see \citet{CGFWSurvey2022} for more details.

We apply \cref{lem:transferBounds} to
$$
p_\ell = \frac{\innp{\nabla f(x_\ell), s_\ell - v_\ell}^2}{4L\norm{s_\ell - v_\ell}^2} \qquad \text{and} \qquad g(x_\ell) = \frac{\innp{\nabla f(x_\ell), s_\ell - v_\ell}^2}{2\mu w(P)^2},
$$
to obtain 
$$
1/A_\ell = \left(1 - \frac{\mu}{L}\frac{w(P)^2}{2D^2}\right) 1/A_{\ell-1},
$$
as expected but again for the primal-dual gap.

\subsubsection{Strongly convex feasible region and $\norm{\nabla f(x)} \geq c > 0$}

We now consider the case where the feasible region is strongly convex and the unconstrained minimizer of $f$ lies outside of the feasible region. Since the feasible region is strongly convex we obtain the scaling inequality \citep{garber2015faster,CGFWSurvey2022}:
$$
\frac{\alpha}{4} \norm{\nabla f(x_t)} \leq \frac{\innp{\nabla f(x_t), x_t - v_t}}{\norm{x_t - v_t}^2}.
$$
Moreover we have 
$$
f(x_t) - f(x^*) \leq \innp{\nabla f(x_t), x_t - v_t}.
$$
Combining the two leads to:
$$
(f(x_t) - f(x^*)) \frac{\alpha}{4} \norm{\nabla f(x_t)} \leq \frac{\innp{\nabla f(x_t), x_t - v_t}^2}{\norm{x_t - v_t}^2},
$$
which gives the dual gap bound:
\begin{equation}
	\label{fw:SCB}
f(x_t) - f(x^*) \leq \frac{4}{\alpha \norm{\nabla f(x_t)}} \frac{\innp{\nabla f(x_t), x_t - v_t}^2}{\norm{x_t - v_t}^2},
\end{equation}
and we set 
$$g(x_t) \doteq \frac{4}{\alpha \norm{\nabla f(x_t)}} \frac{\innp{\nabla f(x_t), x_t - v_t}^2}{\norm{x_t - v_t}^2}.$$
Moreover, for the primal progress bound we simply pick the short-step with the standard direction $x_t - v_t$, i.e.,
$$
p_t \doteq \frac{\innp{\nabla f(x_\ell), x_\ell - v_\ell}^2}{2L\norm{x_\ell - v_\ell}^2}.
$$
Applying \cref{lem:transferBounds} we immediately obtain:
$$
	1/A_\ell = \left(1 - \frac{\alpha \norm{\nabla f(x_t)}}{8L}\right) 1/A_{\ell-1},
$$
and using $\norm{\nabla f(x)}\geq c > 0$ we have 
$$
1/A_\ell \leq \left(1 - \frac{\alpha c}{8L}\right) 1/A_{\ell-1},
$$
which completes the argument.

\subsection{Variants utilizing Bregman Divergences: the case where $E_\ell \neq 0$}

We will now consider the slightly more involved case where $E_\ell \neq 0$. This means that we cannot simply apply \cref{lem:transferBounds} but rather have to account for the error term $E_\ell$. To this end, we first collect alternative choices of $a_t$, $A_t$, and $\gamma_t$, assuming linear coupling $\gamma_t = a_t / A_t$, that will become useful in the rest of the section.

\begin{lemma}[Open-loop step-sizes]
\label{lem:open-loop-steps}
Let $1 \leq \ell \in \mathbb N$ and define $\gamma_t = \frac{\ell}{t + \ell}$. Then 
$$a_t = a_0 \binom{t + \ell - 1}{t} \qquad{and}\qquad A_t = a_0 \binom{t + \ell}{t},$$
is a valid choice and moreover it holds:
\begin{enumerate}
\item $A_t = a_t \frac{t + \ell}{\ell}$;
\item $\frac{a_t^2}{A_t} = \gamma_t a_t = \frac{\ell}{t + \ell} a_t$;
\item With the choice $a_0 = \ell!$, we recover for $\ell = 2$ the known values $a_t = 2 (t+2)$ and $A_t = (t+1) (t+2)$.
\end{enumerate}
\begin{proof}
Verified by straightforward calculation.
\end{proof}
\end{lemma}

The next lemma is folklore and we include it for completeness.

\begin{lemma}[Power Sum Estimations] The following holds:
\begin{enumerate}
    \item
$$\sum_{t = 0}^T t^\alpha 
\begin{cases*}
    \approx \ln(T) + \kappa \leq \ln(T) + 1 & for $\alpha = -1$\\
    = T + 1 & for $\alpha = 0$ \\
    \approx \frac{T^{\alpha + 1}}{\alpha + 1} + \frac{T^{\alpha}}{2}  + O(T^{\alpha - 1}) & for $\alpha > 0$ \\
\end{cases*}
,$$
where $\kappa \approx 0.577$ is the Euler-Mascheroni constant. 
\item 
$$
\int_0^T x^\alpha dx = \frac{T^{\alpha+1}}{\alpha+1}
$$
\item 
$$
\sum_{t = 0}^T t^\alpha \leq \int_0^{T+1} x^\alpha dx = \frac{(T+1)^{\alpha+1}}{\alpha+1}
$$

\end{enumerate}
\begin{proof}
The approximation is via Faulhaber's formula.
\end{proof}
\end{lemma}

As an example, in this section we specifically consider convergence rates arising from the so-called strong growth property as studied in \citet{pena2023affine} and \citet{WPP2023} for \FW{} algorithms. Our starting point is the following simple Bregman expansion of $f$: 

$$D_f(y,x) = f(y) - f(x) - \innp{\nabla f(x), y - x}.$$

We use the standard \FW{} gap (or alternatively the Heavy-Ball variant) as a lower bound as before, so that the critical estimation becomes, with the usual $x_{t+1} = (1-\gamma_t) x_t + \gamma_t v_t$: 
\begin{align*}
    &\ A_t(f(x_{t+1}) - f(x_t)) - a_t \innp{\nabla f(x_t), v_t - x_t} \\
= &\ A_t (\innp{\nabla f(x_t), x_{t+1} - x_t} + D_f(x_{t+1}, x_t)) - a_t \innp{\nabla f(x_t), v_t - x_t}.
\end{align*}

Instead of using the standard linear coupling $\gamma_t = a_t / A_t$, we use the modified linear coupling $\gamma_t = 2 a_t / A_t$, so that we obtain:
\begin{align}
&\ A_t (\innp{\nabla f(x_t), x_{t+1} - x_t} + D_f(x_{t+1}, x_t)) - a_t \innp{\nabla f(x_t), v_t - x_t} \\
= &\ A_t D_f(x_{t+1}, x_t) + 2 a_t \innp{\nabla f(x_t), v_t - x_t} -  a_t \innp{\nabla f(x_t), v_t - x_t} \\
\label{eq:bderror}
= &\ A_t D_f(x_{t+1}, x_t) - a_t \innp{\nabla f(x_t), x_t - v_t}.
\end{align}

The \emph{$(M,r)$-strong growth property} asserts
$$
D_f(x_{t+1}, x_t) \leq \gamma_t^2 M/2 \innp{\nabla f(x_t), x_t - v_t}^r,
$$
with $M > 0$ and $r \in [0,1]$. 

\subsubsection{Primal-Dual Convergence $r = 1$} 

We start with the most obvious case $r = 1$, then \eqref{eq:bderror} becomes:

\begin{align*}
    A_t D_f(x_{t+1}, x_t) - a_t \innp{\nabla f(x_t), x_t - v_t} \leq &\ 2M \frac{a_t^2}{A_t} \innp{\nabla f(x_t), x_t - v_t} - a_t \innp{\nabla f(x_t), x_t - v_t} \\
= &\ \innp{\nabla f(x_t), x_t - v_t} a_t (2M \frac{a_t}{A_t} - 1),
\end{align*}
and hence choosing $a_t / A_t = \frac{1}{2M}$, which implies $\gamma_t = \frac{1}{M}$, the right-hand side becomes $0$ and we obtain linear convergence with a rate $1/A_{t-1} (1-\frac{1}{2M}) = 1 / A_t$. 

\subsubsection{Primal-Dual Convergence $r \in [0,1)$}

The situation is much more complicated for $r \in (0,1)$. In this case \eqref{eq:bderror} becomes:

\begin{align*}
    A_t D_f(x_{t+1}, x_t) - a_t \innp{\nabla f(x_t), x_t - v_t} \leq &\ 2M \frac{a_t^2}{A_t} \innp{\nabla f(x_t), x_t - v_t}^r - a_t \innp{\nabla f(x_t), x_t - v_t}
\end{align*}

We now apply Young's inequality with $p = \frac{1}{r}$, $q = \frac{1}{1-r}$, rebalancing with $\frac{a_t^r}{r^r}$, so that we obtain: 
 
\begin{align*}
	&\ 2M \frac{a_t^2}{A_t} \innp{\nabla f(x_t), x_t - v_t}^r - a_t \innp{\nabla f(x_t), x_t - v_t} \\
	= &\ 2M \frac{a_t^2}{A_t} \frac{r^r}{a_t^r} \frac{a_t^r}{r^r} \innp{\nabla f(x_t), x_t - v_t}^r - a_t \innp{\nabla f(x_t), x_t - v_t} \\
	\leq &\ \frac{\left(2M \frac{a_t^2}{A_t} \frac{r^r}{a_t^r}\right)^q}{q} + \frac{\left(\frac{a_t^r}{r^r} \innp{\nabla f(x_t), x_t - v_t}^r\right)^{1/r}}{1/r} - a_t \innp{\nabla f(x_t), x_t - v_t} \\
	= &\ (1-r) \left(2M \frac{a_t^2}{A_t} \frac{r^r}{a_t^r}\right)^\frac{1}{1-r} 
	= (1-r) \left(2M r^r\right)^\frac{1}{1-r} \left(\frac{a_t^{2-r}}{A_t} \right)^\frac{1}{1-r}
\end{align*}

Summing the errors as customary, we obtain:

\begin{align*}
	\frac{1}{A_T} \sum_{\ell = 0}^T E_\ell \leq &\ (1-r) \left(2M r^r\right)^\frac{1}{1-r} \frac{1}{A_T} \sum_{t = 0}^T \left(\frac{a_t^{2-r}}{A_t} \right)^\frac{1}{1-r},
\end{align*}
which now needs to be estimated to derive the convergence rate.

\paragraph{Estimation in the order}
We will first do an approximate estimation, to recover the order of convergence and provide intuition. Suppose that $a_t = t^c$ and hence $A_t \approx \frac{t^{c+1}}{c+1} = \Theta(t^{c+1})$. With this we obtain: 

\begin{align*}
	&\ (1-r) \left(2M r^r\right)^\frac{1}{1-r} \frac{c+1}{T^{c+1}} \sum_{t = 0}^T \left(\frac{t^{c(2-r)}}{t^{c+1}} \right)^\frac{1}{1-r} \\
	= &\ (1-r) \left(2M r^r\right)^\frac{1}{1-r} \frac{c+1}{T^{c+1}} \sum_{t = 0}^T \left(t^{c(2-r) - c -1} \right)^\frac{1}{1-r} \\
	\approx &\ (1-r) \left(2M r^r\right)^\frac{1}{1-r} \frac{c+1}{T^{c+1}} \frac{T^{c - \frac{1}{1-r} + 1}}{c - \frac{1}{1-r} + 1} = (1-r) \left(2M r^r\right)^\frac{1}{1-r} \frac{c+1}{c - \frac{1}{1-r} + 1} T^{- \frac{1}{1-r}},
\end{align*}
assuming $c > \frac{1}{1-r}$. Therefore we obtain for the total convergence rate:
$$G_t \leq \frac{A_0 G_0}{A_T} + (1-r) \left(2M r^r\right)^\frac{1}{1-r} \frac{c+1}{c - \frac{1}{1-r} + 1} T^{- \frac{1}{1-r}},$$
together with $1/A_T \approx \frac{c+1}{T^{c+1}} \leq \frac{c+1}{T^{\frac{1}{1-r}+1}}$, implies via a rather weak estimation
$$G_t \leq 2 (1-r) \left(2M r^r\right)^\frac{1}{1-r} \frac{c+1}{c - \frac{1}{1-r} + 1} T^{- \frac{1}{1-r}},$$

\paragraph{Specific choice of $a_t$} We will now refine the analysis from above via a specific choice of $a_t$ and $A_t$. We choose $a_t = a_0 \binom{t + \ell - 1}{t}$ and $A_t = a_0 \binom{t + \ell}{t}$ similar to \cref{lem:open-loop-steps}. Note however the changed linear coupling $2 a_t / A_t = \gamma_t$ that we use here:

\begin{align*}
	&\ (1-r) \left(2M r^r\right)^\frac{1}{1-r} \frac{1}{A_T} \sum_{t = 0}^T \left(\frac{a_t^{2-r}}{A_t} \right)^\frac{1}{1-r} \\
	= &\ (1-r) \left(2M r^r\right)^\frac{1}{1-r} \frac{1}{A_T} \sum_{t = 0}^T \left(\frac{a_t}{A_t} \right)^\frac{1}{1-r} a_t \\
	= &\ (1-r) \left(2M r^r\right)^\frac{1}{1-r} a_0^{-1} \binom{T + \ell}{T}^{-1} \sum_{t = 0}^T \left(\frac{\ell}{t + \ell}\right)^\frac{1}{1-r} a_0 \binom{t + \ell - 1}{t} \\
	= &\ (1-r) \left(2M r^r\right)^\frac{1}{1-r} \binom{T + \ell}{T}^{-1} \sum_{t = 0}^T \left(\frac{\ell}{t + \ell}\right)^\frac{1}{1-r} \binom{t + \ell - 1}{t} \\
    \leq &\ (1-r) \left(2M r^r\right)^\frac{1}{1-r} \binom{T + \ell}{T}^{-1} \frac{\ell^\frac{1}{1-r}}{(\ell - 1)!} \sum_{t = 0}^T 
    t^{\ell - 1 - \frac{1}{1-r}} \\
    = &\ (1-r) \left(2M r^r\right)^\frac{1}{1-r} \binom{T + \ell}{T}^{-1} \frac{\ell^\frac{1}{1-r}}{(\ell - 1)!} \left(\frac{T^{\ell - \frac{1}{1-r}}}{\ell - \frac{1}{1-r}} + O\left(T^{\ell - 1 - \frac{1}{1-r}}\right)\right) \\
    \leq &\ (1-r) \left(2M r^r\right)^\frac{1}{1-r} 
    \frac{\ell^\ell}{T^\ell}
    \frac{\ell^\frac{1}{1-r}}{(\ell - 1)!} \left(\frac{T^{\ell - \frac{1}{1-r}}}{\ell - \frac{1}{1-r}} + O\left(T^{\ell - 1 - \frac{1}{1-r}}\right)\right) \\
    = &\ (1-r) \left(2M r^r\right)^\frac{1}{1-r} 
    \frac{\ell^{\ell + \frac{1}{1-r}}}{(\ell - 1)!} \left(\frac{T^{- \frac{1}{1-r}}}{\ell - \frac{1}{1-r}} + O\left(T^{- 1 - \frac{1}{1-r}}\right)\right) \\
    = &\ (1-r) \left(2M r^r\right)^\frac{1}{1-r} 
    \frac{\ell^{\ell + \frac{1}{1-r}}}{(\ell - \frac{1}{1-r})(\ell - 1)!} T^{- \frac{1}{1-r}} + O\left(T^{- 1 - \frac{1}{1-r}}\right),
    \end{align*}
with $\ell > 1 + \frac{1}{1-r}$. Therefore we obtain for the total convergence rate:
$$G_t \leq \frac{A_0 G_0}{A_T} + (1-r) \left(2M r^r\right)^\frac{1}{1-r}  \frac{\ell^{\ell + \frac{1}{1-r}}}{(\ell - \frac{1}{1-r})(\ell - 1)!} T^{- \frac{1}{1-r}} + O\left(T^{- 1 - \frac{1}{1-r}}\right),$$
together with $1/A_T \leq \ell^\ell T^{-\ell} \leq \ell^\ell T^{- 1 - \frac{1}{1-r}}$, this implies 
$$(1-r) \left(2M r^r\right)^\frac{1}{1-r}  \frac{\ell^{\ell + \frac{1}{1-r}}}{(\ell - \frac{1}{1-r})(\ell - 1)!} T^{- \frac{1}{1-r}} + O\left(T^{- 1 - \frac{1}{1-r}}\right).$$

\begin{remark}[Step-size choice and short steps]
Note,that we can always use short-steps instead of the $\frac{\ell}{t+\ell}$-step size by monotonicity and obtain the same rates.
\end{remark}

\section{Experiments: Appendix}
\label{sec:experiments_appendix}

We performed additional experiments to test convergence and isolate the effect of the different components. Note that for the additional experiments below, we sometimes have used additional \LMO{} calls for computing evaluation metrics for comparisons, so that the timings in the right columns are not as reliable and not comparable to the main experiments. Either way, the performance in iterations gives a clear picture here as the cost per iteration of all algorithms is quite comparable, apart from the adaptive variant which behaves more like line search but serves as a baseline only.

\subsection*{Primal-dual short-steps}

We conducted a series of experiments to evaluate the performance of the primal-dual short step method described in \cref{sec:primal-dual-short-step}. Unfortunately, our findings indicate that this variant does not outperform the vanilla \FW{} variant with short-steps (or line search). The primary reason for this effect appears to be the improved primal-dual gap measure, which results in a \emph{smaller} step size, simply by how the step size is computed. Consequently, this leads to slower convergence rates. This behavior is reminiscent of the lower bound discussed in \citep{guelat1986some}, and thus, it is not entirely unexpected. Moreover, by focusing on large primal-dual progress, the algorithm can overshoot in terms of primal progress, so that while the primal-dual progress is relatively large, we are still converging slowly (assuming that the primal-dual gap converges at a certain rate). 

To illustrate this point, we present a typical experiment in \cref{fig:pdss-sq-prob-simplex}. This figure is representative of the results we observed across various tests, consistently showing that the primal-dual short step method converges at a similar rate to the vanilla \FW{} variant. For completeness, we also include the results for the optimistic variant, which indeed outperforms the other variants.

\subsection*{Heavy ball lower bound vs optimism}

In our experiments, we also investigated whether the heavy ball lower bound is inherently stronger than the vanilla \FW{} lower bound, which might have implied that optimism does not actually help at all. However, our findings indicate that this is not the case; see \cref{fig:sfw-nsq-prob-simplex} for a representative experiment. Specifically, when the heavy ball lower bound is applied to the vanilla \FW{} variant, it actually performs worse than the standard \FW{} gap in our tests. This observation strongly suggests that the superior performance of the optimistic variant is due to its trajectory rather than the strength of the lower bound itself. 

\subsection*{Additional experiments}
We present two additional experiments in \cref{fig:nsq_optimistic_fw_oftrl_prob_simplex,fig:sq_optimistic_fw_oftrl_k-sparse} that had to be relegated to the appendix due to space constraints. Also here the results are consistent with our findings in the main text.

\begin{figure*}[ht]
    \begin{center}
    \includegraphics[width=0.8\textwidth]{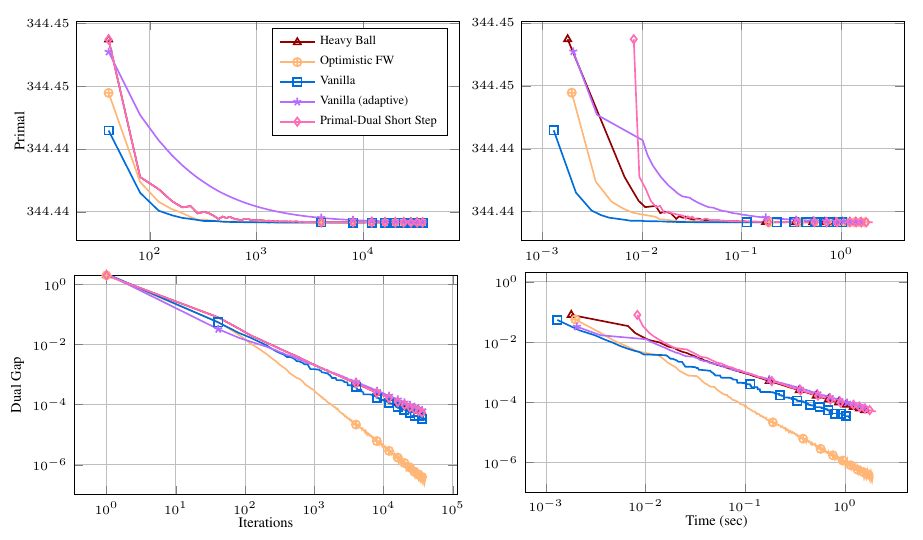}
    \end{center}
    \caption{\label{fig:pdss-sq-prob-simplex} Comparison over the probability simplex of dimension $n=1000$ with objective $f(x) = \norm{x - x_0}_2^2$, where $x_0$ is a random point outside the probability simplex. We see that primal-dual short steps converge exactly as fast as the vanilla variant.}
\end{figure*}

\begin{figure*}[ht]
    \begin{center}
    \includegraphics[width=0.8\textwidth]{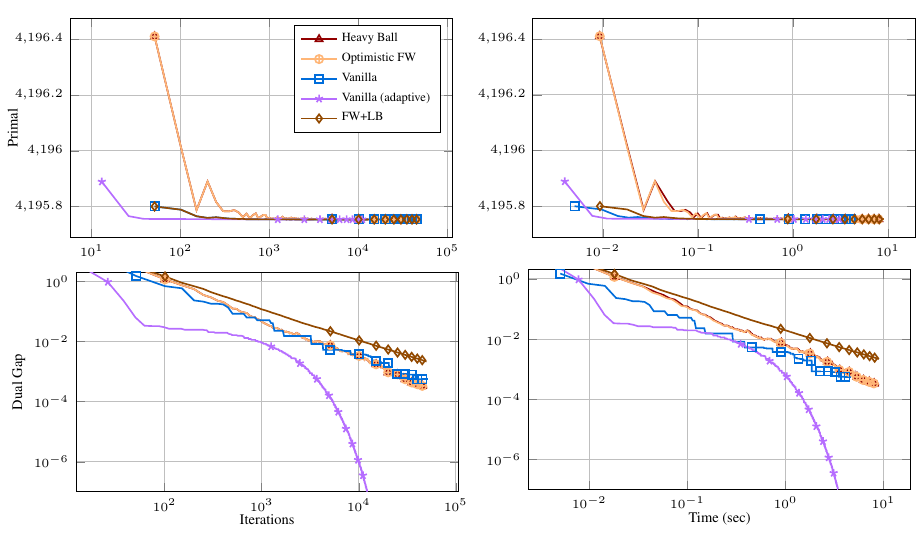}
    \end{center}
    \caption{\label{fig:sfw-nsq-prob-simplex} Comparison over the probability simplex of dimension $n=1000$ with objective $f(x) = \norm{x - x_0}_2^2$, where $x_0$ is a random point outside the probability simplex. FW+LB denotes the standard \FW{} variant with the heavy ball lower bound. We can see that FW with the heavy ball lower bound converges no faster than the normal FW in dual gap.}
\end{figure*}

\begin{figure*}[ht]
    \begin{center}
    \includegraphics[width=0.8\textwidth]{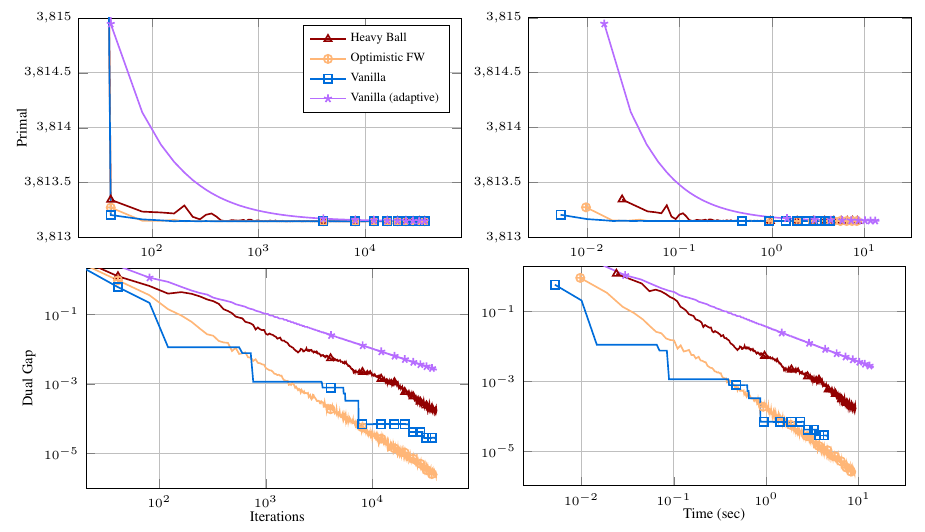}
    \end{center}
    \caption{\label{fig:nsq_optimistic_fw_oftrl_prob_simplex} Comparison over the probability simplex of dimension $n = 100$ with objective $f(x) = \norm{Ax - b}_2^2$, where $A$ and $b$ are random. We can see that the optimistic variant converges faster than the other variants both in iterations and time.}
\end{figure*}

\begin{figure*}[ht]
    \begin{center}
    \includegraphics[width=0.8\textwidth]{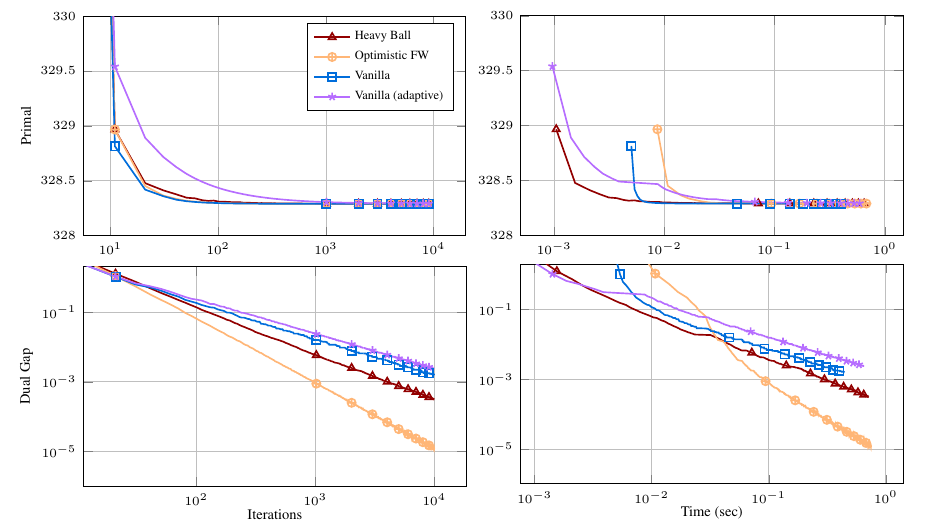}
    \end{center}
    \caption{\label{fig:sq_optimistic_fw_oftrl_k-sparse} Comparison over $k$-sparse polytope of dimension $n=1000$ and $k = 10$ with objective $f(x) = \norm{x - x_0}_2^2$, where $x_0$ is a random point outside the $k$-sparse polytope. We can see that the optimistic variant converges faster than the other variants both in iterations and time.}
\end{figure*}

\end{document}

%% file: algorithm_config.tex
\newcounter{myalg}
\AtBeginEnvironment{algorithmic}{\refstepcounter{myalg}}
\makeatletter
\@addtoreset{ALC@unique}{myalg}
\makeatother

\crefname{ALC@unique}{Line}{Lines}

\newsavebox{\algleft}
\newsavebox{\algright}

%% file: main.bbl
\begin{thebibliography}{45}
\providecommand{\natexlab}[1]{#1}
\providecommand{\url}[1]{\texttt{#1}}
\expandafter\ifx\csname urlstyle\endcsname\relax
  \providecommand{\doi}[1]{doi: #1}\else
  \providecommand{\doi}{doi: \begingroup \urlstyle{rm}\Url}\fi

\bibitem[Abernethy \& Wang(2017)Abernethy and Wang]{abernethy2017fwequilibrium}
Abernethy, J.~D. and Wang, J.
\newblock On frank-wolfe and equilibrium computation.
\newblock In Guyon, I., von Luxburg, U., Bengio, S., Wallach, H.~M., Fergus,
  R., Vishwanathan, S. V.~N., and Garnett, R. (eds.), \emph{Advances in Neural
  Information Processing Systems 30: Annual Conference on Neural Information
  Processing Systems 2017, December 4-9, 2017, Long Beach, CA, {USA}}, pp.\
  6584--6593, 2017.
\newblock URL
  \url{https://proceedings.neurips.cc/paper/2017/hash/7371364b3d72ac9a3ed8638e6f0be2c9-Abstract.html}.

\bibitem[Azoury \& Warmuth(2001)Azoury and Warmuth]{azoury2011relative}
Azoury, K.~S. and Warmuth, M.~K.
\newblock Relative loss bounds for on-line density estimation with the
  exponential family of distributions.
\newblock \emph{Mach. Learn.}, 43\penalty0 (3):\penalty0 211--246, 2001.
\newblock \doi{10.1023/A:1010896012157}.
\newblock URL \url{https://doi.org/10.1023/A:1010896012157}.

\bibitem[Bach(2021)]{bach2021effectiveness}
Bach, F.
\newblock On the effectiveness of richardson extrapolation in data science.
\newblock \emph{SIAM Journal on Mathematics of Data Science}, 3\penalty0
  (4):\penalty0 1251--1277, 2021.

\bibitem[Bach(2015)]{bach2015duality}
Bach, F.~R.
\newblock Duality between subgradient and conditional gradient methods.
\newblock \emph{{SIAM} J. Optim.}, 25\penalty0 (1):\penalty0 115--129, 2015.
\newblock \doi{10.1137/130941961}.
\newblock URL \url{https://doi.org/10.1137/130941961}.

\bibitem[Bellet et~al.(2015)Bellet, Liang, Garakani, Balcan, and
  Sha]{bellet2015distributed}
Bellet, A., Liang, Y., Garakani, A.~B., Balcan, M., and Sha, F.
\newblock A distributed frank-wolfe algorithm for communication-efficient
  sparse learning.
\newblock In Venkatasubramanian, S. and Ye, J. (eds.), \emph{Proceedings of the
  2015 {SIAM} International Conference on Data Mining, Vancouver, BC, Canada,
  April 30 - May 2, 2015}, pp.\  478--486. {SIAM}, 2015.
\newblock \doi{10.1137/1.9781611974010.54}.
\newblock URL \url{https://doi.org/10.1137/1.9781611974010.54}.

\bibitem[Braun et~al.(2022)Braun, Carderera, Combettes, Hassani, Karbasi,
  Mokthari, and Pokutta]{CGFWSurvey2022}
Braun, G., Carderera, A., Combettes, C.~W., Hassani, H., Karbasi, A., Mokthari,
  A., and Pokutta, S.
\newblock Conditional gradient methods.
\newblock \emph{{preprint}}, 11 2022.

\bibitem[Chen et~al.(2020)Chen, Zhou, Yi, and Gu]{chen2020fwframework}
Chen, J., Zhou, D., Yi, J., and Gu, Q.
\newblock A frank-wolfe framework for efficient and effective adversarial
  attacks.
\newblock In \emph{The Thirty-Fourth {AAAI} Conference on Artificial
  Intelligence, {AAAI} 2020, The Thirty-Second Innovative Applications of
  Artificial Intelligence Conference, {IAAI} 2020, The Tenth {AAAI} Symposium
  on Educational Advances in Artificial Intelligence, {EAAI} 2020, New York,
  NY, USA, February 7-12, 2020}, pp.\  3486--3494. {AAAI} Press, 2020.
\newblock \doi{10.1609/AAAI.V34I04.5753}.
\newblock URL \url{https://doi.org/10.1609/aaai.v34i04.5753}.

\bibitem[Chiang et~al.(2012)Chiang, Yang, Lee, Mahdavi, Lu, Jin, and
  Zhu]{chiang2012online}
Chiang, C., Yang, T., Lee, C., Mahdavi, M., Lu, C., Jin, R., and Zhu, S.
\newblock Online optimization with gradual variations.
\newblock In Mannor, S., Srebro, N., and Williamson, R.~C. (eds.), \emph{{COLT}
  2012 - The 25th Annual Conference on Learning Theory, June 25-27, 2012,
  Edinburgh, Scotland}, volume~23 of \emph{{JMLR} Proceedings}, pp.\
  6.1--6.20. JMLR.org, 2012.
\newblock URL \url{http://proceedings.mlr.press/v23/chiang12/chiang12.pdf}.

\bibitem[Combettes \& Pokutta(2020)Combettes and Pokutta]{CP2020boost}
Combettes, C.~W. and Pokutta, S.
\newblock {Boosting Frank-Wolfe by Chasing Gradients}.
\newblock \emph{{Proceedings of ICML}}, 3 2020.

\bibitem[Cutkosky(2019)]{cutkosky2019anytime}
Cutkosky, A.
\newblock Anytime online-to-batch, optimism and acceleration.
\newblock In Chaudhuri, K. and Salakhutdinov, R. (eds.), \emph{Proceedings of
  the 36th International Conference on Machine Learning, {ICML} 2019, 9-15 June
  2019, Long Beach, California, {USA}}, volume~97 of \emph{Proceedings of
  Machine Learning Research}, pp.\  1446--1454. {PMLR}, 2019.
\newblock URL \url{http://proceedings.mlr.press/v97/cutkosky19a.html}.

\bibitem[Diakonikolas \& Orecchia(2019)Diakonikolas and
  Orecchia]{diakonikolas2019approximate}
Diakonikolas, J. and Orecchia, L.
\newblock The approximate duality gap technique: {A} unified theory of
  first-order methods.
\newblock \emph{{SIAM} J. Optim.}, 29\penalty0 (1):\penalty0 660--689, 2019.
\newblock \doi{10.1137/18M1172314}.
\newblock URL \url{https://doi.org/10.1137/18M1172314}.

\bibitem[Diakonikolas et~al.(2020)Diakonikolas, Carderera, and
  Pokutta]{CDP2019}
Diakonikolas, J., Carderera, A., and Pokutta, S.
\newblock {Locally Accelerated Conditional Gradients}.
\newblock \emph{{Proceedings of AISTATS}}, 2020.

\bibitem[Frank \& Wolfe(1956)Frank and Wolfe]{frank1956algorithm}
Frank, M. and Wolfe, P.
\newblock An algorithm for quadratic programming.
\newblock \emph{Naval research logistics quarterly}, 3\penalty0 (1-2):\penalty0
  95--110, 1956.

\bibitem[Garber \& Hazan(2015)Garber and Hazan]{garber2015faster}
Garber, D. and Hazan, E.
\newblock Faster rates for the frank-wolfe method over strongly-convex sets.
\newblock In \emph{International Conference on Machine Learning}, pp.\
  541--549. PMLR, 2015.

\bibitem[Gu{\'e}lat \& Marcotte(1986)Gu{\'e}lat and Marcotte]{guelat1986some}
Gu{\'e}lat, J. and Marcotte, P.
\newblock Some comments on {W}olfe's ‘away step’.
\newblock \emph{Mathematical Programming}, 35\penalty0 (1):\penalty0 110--119,
  1986.

\bibitem[Gutman \& Pe{\~{n}}a(2023)Gutman and Pe{\~{n}}a]{gutman2023perturbed}
Gutman, D.~H. and Pe{\~{n}}a, J.~F.
\newblock Perturbed fenchel duality and first-order methods.
\newblock \emph{Math. Program.}, 198\penalty0 (1):\penalty0 443--469, 2023.
\newblock \doi{10.1007/S10107-022-01779-7}.
\newblock URL \url{https://doi.org/10.1007/s10107-022-01779-7}.

\bibitem[Hendrych et~al.(2024)Hendrych, Besan\c{c}on, and
  Pokutta]{hendrych2023solving}
Hendrych, D., Besan\c{c}on, M., and Pokutta, S.
\newblock Solving the optimal experiment design problem with mixed-integer
  convex methods.
\newblock In Liberti, L. (ed.), \emph{22nd International Symposium on
  Experimental Algorithms (SEA 2024)}, volume 301 of \emph{Leibniz
  International Proceedings in Informatics (LIPIcs)}, pp.\  16:1--16:22,
  Dagstuhl, Germany, 2024. Schloss Dagstuhl -- Leibniz-Zentrum f{\"u}r
  Informatik.
\newblock ISBN 978-3-95977-325-6.
\newblock \doi{10.4230/LIPIcs.SEA.2024.16}.
\newblock URL
  \url{https://drops.dagstuhl.de/entities/document/10.4230/LIPIcs.SEA.2024.16}.

\bibitem[Jaggi(2013)]{jaggi2013revisiting}
Jaggi, M.
\newblock Revisiting frank-wolfe: Projection-free sparse convex optimization.
\newblock In \emph{Proceedings of the 30th International Conference on Machine
  Learning, {ICML} 2013, Atlanta, GA, USA, 16-21 June 2013}, volume~28 of
  \emph{{JMLR} Workshop and Conference Proceedings}, pp.\  427--435. JMLR.org,
  2013.
\newblock URL \url{http://proceedings.mlr.press/v28/jaggi13.html}.

\bibitem[Kerdreux et~al.(2021)Kerdreux, d'Aspremont, and Pokutta]{KAP2021}
Kerdreux, T., d'Aspremont, A., and Pokutta, S.
\newblock {Local and Global Uniform Convexity Conditions}.
\newblock \emph{preprint}, 2 2021.

\bibitem[Lacoste-Julien \& Jaggi(2015)Lacoste-Julien and
  Jaggi]{lacoste2015global}
Lacoste-Julien, S. and Jaggi, M.
\newblock On the global linear convergence of frank-wolfe optimization
  variants.
\newblock \emph{Advances in neural information processing systems}, 28, 2015.

\bibitem[Lacoste{-}Julien et~al.(2015)Lacoste{-}Julien, Lindsten, and
  Bach]{lacoste2015sequential}
Lacoste{-}Julien, S., Lindsten, F., and Bach, F.~R.
\newblock Sequential kernel herding: Frank-wolfe optimization for particle
  filtering.
\newblock In Lebanon, G. and Vishwanathan, S. V.~N. (eds.), \emph{Proceedings
  of the Eighteenth International Conference on Artificial Intelligence and
  Statistics, {AISTATS} 2015, San Diego, California, USA, May 9-12, 2015},
  volume~38 of \emph{{JMLR} Workshop and Conference Proceedings}. JMLR.org,
  2015.
\newblock URL \url{http://proceedings.mlr.press/v38/lacoste-julien15.html}.

\bibitem[Levitin \& Polyak(1966)Levitin and Polyak]{levitin1966constrained}
Levitin, E.~S. and Polyak, B.~T.
\newblock Constrained minimization methods.
\newblock \emph{USSR Computational mathematics and mathematical physics},
  6\penalty0 (5):\penalty0 1--50, 1966.

\bibitem[Li et~al.(2021)Li, Sadeghi, and Giannakis]{li2021heavy}
Li, B., Sadeghi, A., and Giannakis, G.~B.
\newblock Heavy ball momentum for conditional gradient.
\newblock In Ranzato, M., Beygelzimer, A., Dauphin, Y.~N., Liang, P., and
  Vaughan, J.~W. (eds.), \emph{Advances in Neural Information Processing
  Systems 34: Annual Conference on Neural Information Processing Systems 2021,
  NeurIPS 2021, December 6-14, 2021, virtual}, pp.\  21244--21255, 2021.
\newblock URL
  \url{https://proceedings.neurips.cc/paper/2021/hash/b166b57d195370cd41f80dd29ed523d9-Abstract.html}.

\bibitem[Lu \& Freund(2021)Lu and Freund]{lu2021generalized}
Lu, H. and Freund, R.~M.
\newblock Generalized stochastic frank-wolfe algorithm with stochastic
  "substitute" gradient for structured convex optimization.
\newblock \emph{Math. Program.}, 187\penalty0 (1):\penalty0 317--349, 2021.
\newblock \doi{10.1007/S10107-020-01480-7}.
\newblock URL \url{https://doi.org/10.1007/s10107-020-01480-7}.

\bibitem[Luise et~al.(2019)Luise, Salzo, Pontil, and
  Ciliberto]{luise2019sinkhorn}
Luise, G., Salzo, S., Pontil, M., and Ciliberto, C.
\newblock Sinkhorn barycenters with free support via frank-wolfe algorithm.
\newblock In Wallach, H.~M., Larochelle, H., Beygelzimer, A.,
  d'Alch{\'{e}}{-}Buc, F., Fox, E.~B., and Garnett, R. (eds.), \emph{Advances
  in Neural Information Processing Systems 32: Annual Conference on Neural
  Information Processing Systems 2019, NeurIPS 2019, December 8-14, 2019,
  Vancouver, BC, Canada}, pp.\  9318--9329, 2019.
\newblock URL
  \url{https://proceedings.neurips.cc/paper/2019/hash/9f96f36b7aae3b1ff847c26ac94c604e-Abstract.html}.

\bibitem[Mokhtari et~al.(2020)Mokhtari, Hassani, and
  Karbasi]{mokhtari2020stochastic}
Mokhtari, A., Hassani, H., and Karbasi, A.
\newblock Stochastic conditional gradient methods: From convex minimization to
  submodular maximization.
\newblock \emph{J. Mach. Learn. Res.}, 21:\penalty0 105:1--105:49, 2020.
\newblock URL \url{https://jmlr.org/papers/v21/18-764.html}.

\bibitem[Mortagy et~al.(2020)Mortagy, Gupta, and Pokutta]{MGP2020walking}
Mortagy, H., Gupta, S., and Pokutta, S.
\newblock {Walking in the Shadow: A New Perspective on Descent Directions for
  Constrained Minimization}.
\newblock \emph{{Proceedings of NeurIPS}}, 6 2020.

\bibitem[N{\'{e}}giar et~al.(2020)N{\'{e}}giar, Dresdner, Tsai, Ghaoui,
  Locatello, Freund, and Pedregosa]{negiar2020stochastic}
N{\'{e}}giar, G., Dresdner, G., Tsai, A.~Y., Ghaoui, L.~E., Locatello, F.,
  Freund, R., and Pedregosa, F.
\newblock Stochastic frank-wolfe for constrained finite-sum minimization.
\newblock In \emph{Proceedings of the 37th International Conference on Machine
  Learning, {ICML} 2020, 13-18 July 2020, Virtual Event}, volume 119 of
  \emph{Proceedings of Machine Learning Research}, pp.\  7253--7262. {PMLR},
  2020.
\newblock URL \url{http://proceedings.mlr.press/v119/negiar20a.html}.

\bibitem[Nesterov \& Shikhman(2015)Nesterov and Shikhman]{nesterov2015quasi}
Nesterov, Y. and Shikhman, V.
\newblock Quasi-monotone subgradient methods for nonsmooth convex minimization.
\newblock \emph{Journal of Optimization Theory and Applications}, 165\penalty0
  (3):\penalty0 917--940, 2015.
\newblock URL
  \url{https://link.springer.com/article/10.1007/s10957-014-0677-5}.

\bibitem[Nesterov(2018)]{nesterov2018complexity}
Nesterov, Y.~E.
\newblock Complexity bounds for primal-dual methods minimizing the model of
  objective function.
\newblock \emph{Math. Program.}, 171\penalty0 (1-2):\penalty0 311--330, 2018.
\newblock \doi{10.1007/S10107-017-1188-6}.
\newblock URL \url{https://doi.org/10.1007/s10107-017-1188-6}.

\bibitem[Nguyen et~al.(2022)Nguyen, Fu, and Wu]{nguyen2022memory}
Nguyen, T., Fu, X., and Wu, R.
\newblock Memory-efficient convex optimization for self-dictionary separable
  nonnegative matrix factorization: {A} frank-wolfe approach.
\newblock \emph{{IEEE} Trans. Signal Process.}, 70:\penalty0 3221--3236, 2022.
\newblock \doi{10.1109/TSP.2022.3177845}.
\newblock URL \url{https://doi.org/10.1109/TSP.2022.3177845}.

\bibitem[Orabona(2019)]{orabona2019modern}
Orabona, F.
\newblock A modern introduction to online learning.
\newblock \emph{CoRR}, abs/1912.13213, 2019.
\newblock URL \url{http://arxiv.org/abs/1912.13213}.

\bibitem[Pedregosa et~al.(2020)Pedregosa, Negiar, Askari, and
  Jaggi]{pedregosa2020linearly}
Pedregosa, F., Negiar, G., Askari, A., and Jaggi, M.
\newblock Linearly convergent frank-wolfe with backtracking line-search.
\newblock In \emph{International conference on artificial intelligence and
  statistics}, pp.\  1--10. PMLR, 2020.

\bibitem[Pe{\~n}a(2023)]{pena2023affine}
Pe{\~n}a, J.~F.
\newblock Affine invariant convergence rates of the conditional gradient
  method.
\newblock \emph{SIAM Journal on Optimization}, 33\penalty0 (4):\penalty0
  2654--2674, 2023.

\bibitem[Pe{\~{n}}a \& Rodr{\'i}guez(2018)Pe{\~{n}}a and
  Rodr{\'i}guez]{pena2018polytope}
Pe{\~{n}}a, J.~F. and Rodr{\'i}guez, D.
\newblock Polytope conditioning and linear convergence of the {F}rank--{W}olfe
  algorithm.
\newblock \emph{Mathematics of Operations Research}, 44\penalty0 (1):\penalty0
  1--18, 2018.
\newblock \doi{10.1287/moor.2017.0910}.

\bibitem[Peña(2019)]{pena2019generalized}
Peña, J.
\newblock Generalized conditional subgradient and generalized mirror descent:
  duality, convergence, and symmetry.
\newblock \emph{arXiv preprint arXiv:1903.00459}, 2019.

\bibitem[Pokutta(2024)]{P2023}
Pokutta, S.
\newblock {The Frank-Wolfe algorithm: a short introduction}.
\newblock \emph{{Jahresbericht der Deutschen Mathematiker-Vereinigung}},
  126:\penalty0 3--35, 1 2024.

\bibitem[Popov(1980)]{popov1980modification}
Popov, L.~D.
\newblock A modification of the arrow-hurwitz method of search for saddle
  points.
\newblock \emph{Mat. Zametki}, 28\penalty0 (5):\penalty0 777--784, 1980.

\bibitem[Rakhlin \& Sridharan(2013)Rakhlin and Sridharan]{rakhlin2013online}
Rakhlin, A. and Sridharan, K.
\newblock Online learning with predictable sequences.
\newblock In Shalev{-}Shwartz, S. and Steinwart, I. (eds.), \emph{{COLT} 2013 -
  The 26th Annual Conference on Learning Theory, June 12-14, 2013, Princeton
  University, NJ, {USA}}, volume~30 of \emph{{JMLR} Workshop and Conference
  Proceedings}, pp.\  993--1019. JMLR.org, 2013.
\newblock URL \url{http://proceedings.mlr.press/v30/Rakhlin13.html}.

\bibitem[Steinhardt \& Liang(2014)Steinhardt and Liang]{steinhardt2014}
Steinhardt, J. and Liang, P.
\newblock Adaptivity and optimism: An improved exponentiated gradient
  algorithm.
\newblock In \emph{Proceedings of the 31th International Conference on Machine
  Learning, {ICML} 2014, Beijing, China, 21-26 June 2014}, volume~32 of
  \emph{{JMLR} Workshop and Conference Proceedings}, pp.\  1593--1601.
  JMLR.org, 2014.
\newblock URL \url{http://proceedings.mlr.press/v32/steinhardtb14.html}.

\bibitem[Wang et~al.(2024)Wang, Abernethy, and Levy]{wang2024noregret}
Wang, J., Abernethy, J.~D., and Levy, K.~Y.
\newblock No-regret dynamics in the fenchel game: a unified framework for
  algorithmic convex optimization.
\newblock \emph{Math. Program.}, 205\penalty0 (1):\penalty0 203--268, 2024.
\newblock \doi{10.1007/S10107-023-01976-Y}.
\newblock URL \url{https://doi.org/10.1007/s10107-023-01976-y}.

\bibitem[Wirth et~al.(2023)Wirth, Kerdreux, and Pokutta]{WKP2022}
Wirth, E., Kerdreux, T., and Pokutta, S.
\newblock {Acceleration of Frank-Wolfe algorithms with open loop step-sizes}.
\newblock \emph{{Proceedings of AISTATS}}, 1 2023.

\bibitem[Wirth et~al.(2024)Wirth, Pe{\~n}a, and Pokutta]{WPP2023}
Wirth, E., Pe{\~n}a, J., and Pokutta, S.
\newblock Accelerated affine-invariant convergence rates of the frank-wolfe
  algorithm with open-loop step-sizes.
\newblock \emph{{Mathematical Programming A}}, 12 2024.

\bibitem[Wolfe(1970)]{wolfe1970convergence}
Wolfe, P.
\newblock Convergence theory in nonlinear programming.
\newblock \emph{Integer and nonlinear programming}, pp.\  1--36, 1970.

\bibitem[Ye et~al.(2020)Ye, Gong, Nie, Zhou, Klivans, and Liu]{ye2020good}
Ye, M., Gong, C., Nie, L., Zhou, D., Klivans, A.~R., and Liu, Q.
\newblock Good subnetworks provably exist: Pruning via greedy forward
  selection.
\newblock In \emph{Proceedings of the 37th International Conference on Machine
  Learning, {ICML} 2020, 13-18 July 2020, Virtual Event}, volume 119 of
  \emph{Proceedings of Machine Learning Research}, pp.\  10820--10830. {PMLR},
  2020.
\newblock URL \url{http://proceedings.mlr.press/v119/ye20b.html}.

\end{thebibliography}
